\documentclass[11pt]{amsart}

\usepackage{graphicx}
\usepackage{stmaryrd}
\usepackage{amssymb}

\usepackage{amsmath,amscd}

\usepackage[T1]{fontenc}
\usepackage{yfonts}

\newcommand{\R}{\mathbb{R}}
\newcommand{\Q}{\mathbb{Q}}

\newcommand{\M}{\mathcal{M}}
\newcommand{\N}{\mathcal{N}}

\newcommand{\A}{\mathcal{A}}

\renewcommand{\P}{\mathbb{P}}

\renewcommand{\Q}{\mathbb{Q}}

\newcommand{\G}{\mathcal{G}}
\newcommand{\K}{\mathcal{K}}
\newcommand{\I}{\mathcal{I}}
\newcommand{\J}{\mathcal{J}}

\newcommand{\Bor}{\mathbf{Bor}}

\newcommand{\har}{\!\!\upharpoonright\!\!}
\renewcommand{\int}{\mbox{int}}
\newcommand{\diam}{\mathrm{diam}}
\newcommand{\dom}{\mbox{dom}}
\newcommand{\cl}[1]{\overline{#1}}
\newcommand{\ocl}{\mathrm{cl}}
\newcommand{\otr}{\omega^{<\omega}}
\newcommand{\ctr}{2^{<\omega}}

\newcommand{\E}{\mathcal{E}}
\newcommand{\T}{\mathbb{T}}

\newcommand{\ana}{\mathbf{\Sigma}^1_1}
\newcommand{\coana}{\mathbf{\Pi}^1_1}
\newcommand{\lana}{\Sigma^1_1}
\newcommand{\lcoana}{\Pi^1_1}
\newcommand{\lbor}{\Delta^1_1}
\newcommand{\shoenfield}{\mathbf{\Sigma}^1_2}
\newcommand{\negshoenfield}{\mathbf{\Pi}^1_2}

\newcommand{\lnegshoenfield}{\Pi^1_2}
\newcommand{\closed}{\mathbf{\Pi}^0_1}
\newcommand{\baire}{\omega^\omega}
\newcommand{\cantor}{2^\omega}
\newcommand{\gdelta}{\mathbf{G}_\delta}
\newcommand{\fsigma}{\mathbf{F}_\sigma}
\newcommand{\ksigma}{\mathbf{K}_\sigma}
\newcommand{\pc}{\mathrm{pc}}
\newcommand{\p}{\mathrm{p}}

\newcommand{\Lev}{\mathrm{Lev}}

\newcommand{\class}{\mathbf{\Gamma}}
\newcommand{\proj}{\mathrm{proj}}
\newcommand{\iform}{I}
\newcommand{\kform}{K}
\newcommand{\btree}{\omega^{<\omega}}
\newcommand{\ctree}{2^{<\omega}}

\newtheorem{theorem}{Theorem}[section]

\newtheorem{lemma}[theorem]{Lemma}
\newtheorem{corollary}[theorem]{Corollary}

\newtheorem{proposition}[theorem]{Proposition}
\newtheorem*{claim}{Claim}

\theoremstyle{definition}
\newtheorem{definition}[theorem]{Definition}

\newtheorem{remark}[theorem]{Remark}

\author{Marcin Sabok}
\thanks{Research supported by MNiSW grant N 201 361836} 
\address{Mathematical Institute,
  Wroc\l aw University, pl. Grunwaldzki $2\slash 4$,
  $50$-$384$ Wroc\l aw, Poland }
\email{sabok@math.uni.wroc.pl}

\title{Forcing, games and families of closed sets}

\begin{document}

\begin{abstract}
  We propose a new, game-theoretic, approach to the
  idealized forcing, in terms of fusion games. This
  generalizes the classical approach to the Sacks and the
  Miller forcing.  For definable ($\coana$ on $\ana$)
  $\sigma$-ideals we show that if a $\sigma$-ideal is
  generated by closed sets, then it is generated by closed
  sets in all forcing extensions.  We also prove an
  infinite-dimensional version of the Solecki dichotomy for
  analytic sets. Among examples, we investigate the
  $\sigma$-ideal $\E$ generated by closed null sets and
  $\sigma$-ideals connected with not piecewise continuous
  functions.
\end{abstract}

\subjclass[2000]{03E15, 28A05, 54H05}

\keywords{proper forcing, $\sigma$-ideals}

\textwidth=13.5cm

\maketitle

\section{Introduction}

Idealized forcing, developed by Zapletal in \cite{Zpl:DSTDF}
and \cite{Zpl:FI}, is a general approach to the forcing used
in descriptive set theory. If $\I$ is a \mbox{$\sigma$-ideal} on a
Polish space $X$, we consider the associated forcing notion
$\P_\I=\Bor(X)\slash\I$ (or an equivalent forcing,
$\Bor(X)\setminus\I$ ordered by inclusion). Among the
well-known examples are the Sacks forcing and the Miller
forcing. The former is associated to the $\sigma$-ideal of
countable subsets of the Cantor space (or any other Polish
space) and the latter is associated to the $\sigma$-ideal of
$\ksigma$ subsets of the Baire space (i.e. those which can
be covered by a countable union of compact sets).

\medskip

Both the countable sets and $\ksigma$ sets form
$\sigma$-ideals generated by closed sets. Zapletal
\cite{Zpl:FI} investigated the forcing arising from
$\sigma$-ideals generated by closed sets and proved the
following (for definition of \textit{continuous reading of
  names} see \cite[Definition 3.1.1]{Zpl:FI})

\begin{theorem}[{Zapletal, \cite[Theorem
    4.1.2]{Zpl:FI}}]\label{genclosed}
  If $\I$ is a $\sigma$-ideal on a Polish space $X$
  generated by closed sets, then the forcing $\P_\I$ is
  proper and has continuous reading of names in the topology
  of $X$.
\end{theorem}

On the other hand, the Sacks and the Miller forcing are
equivalent to forcings with trees (perfect or superperfect
trees, respectively). In both of these cases we have fusion
(Axiom A), which implies both properness and continuous
reading of names.

In this paper, we generalize this as follows.

\begin{theorem}\label{axiomaclosed}
  If $\I$ is a $\sigma$-ideal on a Polish space $X$
  generated by closed sets, then the forcing $\P_\I$ is
  equivalent to a forcing with trees with the fusion
  property.
\end{theorem}

Axiom A alone can be also deduced from a result of Ishiu
\cite[Theorem B]{Ishiu} and a strenghtening of Theorem
\ref{genclosed} saying that $\P_\I$ is $<\!\omega_1$-proper
\cite[Lemma 1.3]{ZplIshiu}. Our result, however, shows how
to introduce additional structure (trees) which gives even
deeper insight on the forcing $\P_\I$. In particular, it can
be used to give an alternative proof of continuous reading
of names.

The proof of Theorem \ref{axiomaclosed} uses a technique of
fusion games, which are generalizations of the Banach--Mazur
game (cf.  \cite[Section 8.H]{Kechris}). Independently, T.
M\'atrai \cite{Matrai} introduced and studied similar games.
In \cite{Matrai} fusion games are applied to prove
infinite-dimensional perfect set theorems.

\medskip

Although in general there is no perfect-set theorem for
$\sigma$-ideals generated by closed sets, the strongest
property of this kind is expressed in the following theorem
of Solecki.

\begin{theorem}[{Solecki, \cite[Theorem
    1]{Sol.Cov}}]\label{Sol.dich}
  If $\I$ is a $\sigma$-ideal on a Polish space $X$
  generated by closed sets, then each analytic set in $X$
  either belongs to $\I$, or else contains an $\I$-positive
  $\gdelta$ set.
\end{theorem}

The above theorem has important forcing consequences. For
example it implies that if $\I$ is generated by closed sets,
then the forcing $\P_\I$ is equivalent to the forcing
$\Q_\I$ with analytic $\I$-positive sets, as $\P_\I$ is
dense in $\Q_\I$.

\medskip

Definability of a $\sigma$-ideal usually relies on the
property called \textit{$\coana$ on $\ana$} (cf.
\cite[Definition 35.9]{Kechris}), which says that for any
analytic set $A\subseteq X^2$, the set $\{x\in X:
A_x\in\I\}$ is coanalytic. It is worth mentioning that by
classical results of Mazurkiewicz \cite[Theorem
29.19]{Kechris} and Arsenin--Kungui \cite[Theorem
18.18]{Kechris}, both the $\sigma$-ideal of countable
subsets of $\cantor$ and the $\sigma$-ideal $\ksigma$ on
$\baire$ are $\coana$ on $\ana$.

It is well-known (see \cite[Theorem 35.38]{Kechris}) that if
$\K$ is a hereditary and coanalytic (in the sense of the
Effros space) family of closed subsets of $X$, then the
$\sigma$-ideal generated by $\K$ is $\coana$ on $\ana$. As
an application of idealized forcing, we show a new proof of
this fact.

\medskip

In \cite[Section 5.1]{Zpl:FI} Zapletal developed a general
theory of iteration for idealized forcing. There are, of
course, some natural restrictions which the $\sigma$-ideal
$\I$ should fulfill if we want to consider the iterations of
$\P_\I$. These restrictions include definability of $\I$ and
the fact that $\P_\I$ remains proper in all forcing
extensions.  In \cite[Definitions 5.1.2 and 5.1.3]{Zpl:FI}
Zapletal defines \textit{iterable} $\sigma$-ideals and
develops the theory of iteration for this class of
$\sigma$-ideals. 

Note that if $\I$ is a $\coana$ on $\ana$ $\sigma$-ideal,
then it makes sense to define $\I^W$ in any model $W$ of ZFC
(containing the parameters of the definition) as the family
of analytic sets satisfying the $\coana$ definition of $\I$.
In this paper we prove the following result.

\begin{theorem}\label{absoluteness}
  Let $\I$ be a $\coana$ on $\ana$ $\sigma$-ideal.  If $\I$
  is generated by closed sets in $V$, then $\I^W$ is
  generated by closed sets in all forcing extensions
  $V\subseteq W$.
\end{theorem}

Notice that as a corollary, we get that if $\I$ is $\coana$
on $\ana$ generated by closed sets, then $\P_\I$ is proper
in all forcing extensions.

\medskip

Kanovei and Zapletal \cite[Theorem 5.1.9]{Zpl:FI} proved
that if $\I$ is iterable and $\coana$ on $\ana$, then for
each $\alpha<\omega_1$, any analytic set $A\subseteq
X^\alpha$ either belongs to $\I^\alpha$ (the $\alpha$-th
Fubini power of $\I$), or else contains a special kind of
Borel $\I^\alpha$-positive set. In this paper we extend this
result and prove an analogue of the Solecki theorem for the
Fubini products of $\coana$ on $\ana$ $\sigma$-ideals
generated by closed sets. Namely, we re-introduce the notion
of $\bar\I$-positive cubes, when $\langle \I_\beta:
\beta<\alpha\rangle$ is a sequence of $\sigma$-ideals,
$\langle X_\beta:\beta<\alpha\rangle$ is a sequence of
Polish spaces, $\I_\beta$ on $X_\beta$ respectively, and
$\alpha$ is a countable ordinal.  (similar notions have been
considered by other authors under various different names,
cf.  \cite{Kanovei}, \cite{CPA} or \cite{Zpl:FI}). We prove
the following (for definitions see Section
\ref{sec:productsanditerations}).

\begin{theorem}\label{iterationgdelta}
  Let $\langle X_n:n<\omega\rangle$ be a sequence of Polish
  spaces and $\bar \I=\langle \I_n:n<\omega\rangle$ be a
  sequence of $\coana$ on $\ana$ $\sigma$-ideals generated
  by closed sets, $\I_n$ on $X_n$, respectively. If
  $A\subseteq\prod_{n<\omega} X_n$ is $\ana$, then
  \begin{itemize}
  \item either $A\in\bigotimes_{n<\omega}\I_n$,
  \item or else $A$ contains an $\bar\I$-positive $\gdelta$
    cube.
  \end{itemize}
\end{theorem}

The proof of the above theorem relies on a device for
parametrizing $\gdelta$ sets in Polish spaces, which is used
to define descriptive complexity of families of $\gdelta$
sets.

\medskip

In the last two sections we study examples of
$\sigma$-ideals generated by closed sets, motivated by
analysis and measure theory.

First, we investigate an example from measure theory. Let
$\E$ denote the $\sigma$-ideal generated by closed null sets
in the Cantor space. $\E$ has been investigated by
Bartoszy\'nski and Shelah \cite{BaSh}.  In Theorem
\ref{nocohene} we show that $\P_\E$ does not add Cohen
reals, which implies that any forcing extension with $\P_\E$
is minimal. In Corollary \ref{coregame} we establish a
fusion game for the $\sigma$-ideal $\E$.

Next we study an example from the theory of real functions.
Let $X$ and $Y$ be Polish. To any function $f:X\rightarrow
Y$ we associate the $\sigma$-ideal $\I^f$ (on $X$) generated
by closed sets on which $f$ is continuous. The
$\sigma$-ideal $\I^f$ is nontrivial if and only if $f$
cannot be decomposed into countably many continuous
functions with closed domains, i.e. $f$ is not
\textit{piecewise continuous}.  Piecewise continuity of
Baire class 1 functions has been studied by several authors
(Jayne and Rogers \cite{JR}, Solecki \cite{Sol.Dec},
Andretta \cite{Andretta}).  In Corollary
\ref{piececonmiller} we show that if $f:X\rightarrow\baire$
is Baire class 1, not piecewise continuous, then the forcing
$\P_{\I^f}$ is equivalent to the Miller forcing.  In
Corollary \ref{corpiececongame} we establish a fusion game
for the $\sigma$-ideal $\I^f$ when
$f:\cantor\rightarrow\cantor$ is Borel.

\medskip

\textbf{Acknowledgements}. I would like to thank Janusz
Pawlikowski for all the helpful remarks and stimulating
discussion.  I am indebted to Stevo Todor\v cevi\'c for
suggesting to generalize Theorem \ref{iterationgdelta} from
Fubini powers of one $\sigma$-ideal to arbitrary Fubini
products. I would also like to thank Sy Friedman, Tam\'as
M\'atrai and Jind\v rich Zapletal for many useful comments.

\section{Notation}

All Polish spaces in this paper are assumed to be
recursively presented.

If $T\subseteq Y^{<\omega}$ is a tree, then we say that $T$
is a \textit{tree on $Y$}. We write $\lim T$ for $\{x\in
Y^\omega:\ \forall n<\omega\ x\har n\in T\}$.  $\Lev_n(T)$
stands for the set of all elements of $T$ which have length
$n$. If $\tau\in T$, then we write $T(\tau)$ for the tree
$\{\sigma\in T:
\sigma\subseteq\tau\,\vee\,\tau\subseteq\sigma\}$. We write
$[\tau]_T$ for $\{x\in\lim T: \tau\subseteq x\}$ (when $T$
is clear from the context, like $\btree$ or $\ctree$, then
we write only $[\tau]$). We say that $\tau\in T$ is a
\textit{stem of $T$} if $T=T(\tau)$ and $\tau$ is maximal
such. We say that $F\subseteq T$ is a \textit{front of $T$}
if $F$ is an antichain in $T$ and for each $x\in\lim T$
there is $n<\omega$ such that $x\har n\in F$.

By a \textit{$\sigma$-ideal} we mean a family
$\I\subseteq\mathcal{P}(X)$ closed under subsets and
countable unions. We say that a set $B\subseteq X$ is
\textit{$\I$-positive} if $B\not\in\I$. A
\textit{$\sigma$-ideal of analytic sets} is a family of
analytic sets closed under analytic subsets and countable
unions. A \textit{$\sigma$-ideal of closed sets} is defined
analogously. If $\I$ is a $\sigma$-ideal and $B$ is an
$\I$-positive set, then we write $\I\har B$ for $\{A\cap B:
A\in\I\}$.

\section{Fusion games}\label{sec:BanachMazur}

In this section we briefly recall basic definitions
concerning infinite games and introduce \textit{fusion
  games} for $\sigma$-ideals. 

By a \textit{game scheme} we mean a set of rules for a
two-player game (the players are called Adam and Eve, and
Adam begins). Formally, a \textit{game scheme} is a pruned
tree $G\subseteq Y^{<\omega}$ for a countable set $Y$, where
the last elements of sequences at even and odd levels are
understood as possible moves of Eve and Adam, respectively.
In particular, in any game scheme the first move is made by
Adam (the moves are numbered by $\omega\setminus\{0\}$).
Nodes of the tree $G$ of even length are called
\textit{partial plays} and elements of $\lim G$ are called
\textit{plays}. Note that partial plays always end with a
move of Eve.

If $\tau$ is a partial play in a game scheme $G$, then by
the \textit{relativized game scheme} $G_\tau$ we mean the
tree $\{\sigma\in Y^{<\omega}: \tau^\smallfrown\sigma\in
G\}$.  The game scheme $G_\tau$ consists of the games which
``continue'' the partial play $\tau$.

A \textit{payoff set} $p$ in a game scheme $G$ is a subset
of $\lim G$.  By a \textit{game} we mean a pair $(G,p)$
where $G$ is a game scheme and $p$ is a payoff set in $G$
(we say that the game $(G,p)$ is \textit{in the game scheme}
$G$). For a game $(G,p)$ we say that Eve \textit{wins} a
play $g\in\lim G$ if $g\in P$. Otherwise we say that Adam
\textit{wins} $g$.

A \textit{strategy for Adam} in a game scheme $G$ is a
subtree $S\subseteq G$ such that
\begin{itemize}
\item for each odd $n\in\omega$ and $\tau\in S$ such that
  $|\tau|=n$ the set of immediate successors of $\tau$ in
  $S$ contains precisely one point,
\item for each even $n\in\omega$ and $\tau\in S$ such that
  $|\tau|=n$ the sets of immediate successors of $\tau$ in
  $S$ and $G$ are equal.
\end{itemize}
\textit{Strategy for Eve} is defined analogously. If $(G,p)$
is a game in the game scheme $G$ and $S$ is a strategy for
Adam in $G$, then we say that $S$ is a \textit{winning
  strategy for Eve} in the game $(G,p)$ if $\lim S\subseteq
p$. \textit{Winning strategy for Adam} is defined
analogously.

\medskip

Recall the classical Banach--Mazur game \cite[Section
8.H]{Kechris} which ``decides'' whether a Borel set is
meager nor not, in terms of existence of a winning strategy
for one of the players.  Now we introduce an abstract notion
of a fusion game which will cover the classical examples as
well as those from Sections \ref{sec:E} and
\ref{sec:Piececon}. Suppose we have a game scheme $G$
together with a family of payoff sets $p(A)$ for each
$A\subseteq X$ such that
\begin{itemize}
\item[(i)] $p(A)\subseteq\lim G$ is Borel, for each Borel set
  $A\subseteq X$,
\item[(ii)] $p(A)\subseteq p(B)$ for each $B\subseteq A$,
\item[(iii)] $p(\bigcup_{n<\omega} A_n)=\bigcap_{n<\omega}
  p(A_n)$ for each sequence $\langle A_n:n<\omega\rangle$.
\end{itemize}
Intuitively, $p(A)$ is such that a winning strategy for Eve
in $(G,p(A))$ ``proves'' that $A$ is ``small''. For each
$A\subseteq X$ the game $G(A)$ is the game in the game
scheme $G$ with the payoff set $p(A)$. We denote by
$G(\cdot)$ the game scheme $G$ together with the function
$p$.  We call $G(\cdot)$ a \textit{fusion scheme} if
\begin{itemize}
\item[(iv)] the moves of Adam code (in a prescribed way, in
  terms of a fixed enumeration of the basis) basic open sets
  $U_n$ such that $\cl{U_{n+1}}\subseteq U_n$ and
  $\diam(U_n)<1\slash n$,
\item[(v)] for each Borel set $A\subseteq X$ and each play
  $g$ in $G(A)$ if Adam wins $g$, then the single point in
  the intersection of $U_n$'s (as above) is in $A$.
\end{itemize}

Notice that if the family of sets $q\subseteq\lim G$ such
that Eve has a winning strategy in the game $(G,q)$ is
closed under countable intersections, then the family of
sets $A\subseteq X$ such that Eve has a winning strategy in
$G(A)$ forms a $\sigma$-ideal (by (iii)).

The idea of considering $\sigma$-ideals defined in terms
of a winning strategy in a game scheme occurs in a paper of
Schmidt \cite{Schmidt} and later in a work of Mycielski
\cite{Mycielski}.

If the family of sets $A\subseteq X$ for which Eve has a
winning strategy in $G(A)$ forms a $\sigma$-ideal $\I$, then
we say that $G(\cdot)$ is a \textit{fusion scheme for $\I$}.

Suppose $X=\lim T$ for some countable tree $T$. Suppose also
that the game scheme $G$ is such that the possible $n$-th
moves of Adam correspond to elements at the $n$-th level of
$T$ (like in (iv), to the basic clopen sets $[\tau]_T$ for
$\tau\in\Lev_n(T)$). Let $G_\I(\cdot)$ be a fusion scheme
for a $\sigma$-ideal $\I$ and let $\tau$ be a partial play
in the game scheme $G_\I$. Let $U$ be the basic clopen set
coded by the last move of Adam in $\tau$. Recall that the
relativized game scheme $(G_\I)_\tau$ consists of the
continuations of $\tau$ in $G_\I$. The game scheme
$(G_\I)_\tau$ together with the function $p_\tau(A)=p(A)\cap
U$ defines a \textit{relativized fusion scheme}. Using the
property (v) we easily get the following.

\begin{proposition}\label{relativegame}
  Let $X=\lim T$, $G_\I$, $\tau$ and $U$ be as above. Let
  $A\subseteq X$. Eve has a winning strategy in
  $(G_\I)_\tau(A)$ if and only if Eve has a winning strategy
  in $G_\I(A\cap U)$.
\end{proposition}

\section{Fusion in the forcing $\P_\I$}\label{sec:treerep}

In this section we give a proof of Theorem
\ref{axiomaclosed}.  The main ingredient here are fusion
schemes for $\sigma$-ideals generated by closed sets.

We now give an informal outline of the proof of Theorem
\ref{axiomaclosed}. The general idea is as follows: having a
$\sigma$-ideal $\I$ generated by closed sets, we find a
fusion scheme $G_\I(\cdot)$ for $\I$ such that the trees of
winning strategies in $G_\I(B)$ (for $B\in\P_\I$) determine
some analytic $\I$-positive sets. Moreover, for each
$B\in\P_\I$ the winning condition for Adam (the complement
of the payoff set) in $G_\I(B)$ is a $\gdelta$ set in $\lim
G_\I$. We consider the forcing with trees of winning
strategies for Adam in the games $G_\I(B)$ (for all
$B\in\P_\I$) and show that it is equivalent to the original
forcing $\P_\I$ (we in fact show that it is equivalent to
the forcing with $\I$-positive $\ana$ sets and then use
Theorem \ref{Sol.dich} to conclude that all three forcings
are equivalent). Using the fact that the winning conditions
in $G_\I(B)$ are the intersections of $\omega$ many open
sets, we define $\omega$-many fronts in the trees of winning
strategies (such that crossing the $n$-th front implies that
the game is in the $n$-th open set).  Now, using these
fronts as analogues of the splitting levels in the perfect
or superperfect trees, we define fusion in the forcing of
winning strategies for Adam.

Although the general idea is based on the above outline, we
will have to additionally modify the games in order to avoid
some determinacy problems. That is, instead of a fusion
scheme for $\I$ and the games $G_\I(B)$ we will use their
unfolded variant. We would like to emphasize that in many
concrete cases of $\sigma$-ideals (like in Sections
\ref{sec:E} or \ref{sec:Piececon}), we can use simpler
fusion schemes and the fusion from Theorem
\ref{axiomaclosed} can be simplified.

\begin{proof}[Proof of Theorem \ref{axiomaclosed}]
  To simplify notation we assume that the underlying space
  $X$ is the Baire space $\baire$. Pick a bijection
  $\rho:\omega\rightarrow\omega\times\omega$. $H_{\omega_1}$
  stadnds for the family of hereditarily countable sets (it
  will be used to make sure that the forcing we define is a
  set).

  \medskip

  We will use the following notation.
  \begin{itemize}
  \item Let $Y$ be an arbitrary set. If
    $\tau\in(\omega\times Y)^{<\omega}$, then by
    $\bar\tau\in\otr$ we denote the sequence of the first
    coordinates of the elements of $\tau$. Suppose $T$ is a
    tree on $\omega\times Y$. The map $\p_Y:\lim
    T\rightarrow\baire$ is defined as follows: if $t\in\lim
    T$ and $t\har n=\tau_n$, then
    $\p_Y(t)=\bigcup_{n<\omega}\bar\tau_n$.
  \item Let $z$ be arbitrary. If $\tau\in(\omega\times
    Y)^{<\omega}$ and $\tau=\langle a_i:i<|\tau|\rangle$,
    then by $\tau^z\in(\omega\times Y\times\{z\})^{<\omega}$
    we denote the sequence $\langle (a_i,z):
    i<|\tau|\rangle$.  If $T$ is a tree on $\omega\times Y$,
    then by $T^z$ we denote the tree $\{\tau^z:\tau\in T\}$
    on $\omega\times Y\times\{z\}$.
  \item If $Y=W\times Z$ and $\tau\in(\omega\times
    Y)^{<\omega}$, then by $\tau_W\in (\omega\times
    W)^{<\omega}$ we denote $\langle \pi_{\omega\times
      W}(a_i): i<|\tau|\rangle$ (where $\pi_{\omega\times
      W}: \omega\times W\times Z\rightarrow\omega\times W$
    is the projection to the first two coordinates). By
    $T_W$ we denote the tree $\{\tau_W: \tau\in T\}$.
  \end{itemize}

  \medskip

  For $Y\in H_{\omega_1}$ and a tree $T$ on $\omega\times Y$
  let $G_\I(Y,T)$ be the game scheme in which
  \begin{itemize}
  \item in his $n$-th turn Adam constructs $\tau_n\in T$
    such that $\tau_n\supsetneq\tau_{n-1}$
    ($\tau_{-1}=\emptyset$),
  \item in her $n$-th turn Eve picks a clopen set $O_n$ in
    $\baire$ such that
    \begin{displaymath}
      \proj[T(\tau_n)]\not\in \I\ \Rightarrow\ O_n\cap\proj[T(\tau_n)]\not\in \I.
    \end{displaymath}
  \end{itemize}
  By the end of a play, Adam and Eve have a sequence of
  closed sets $E_k$ in $\baire$ defined as follows:
  $$E_k=\cantor\setminus\bigcup_{i<\omega}
  O_{\rho^{-1}(i,k)}.$$ Put
  $x=\bigcup_{n<\omega}\bar\tau_n\in\baire$. Consider a
  payoff set in $G_\I(Y,T)$ such that Adam wins if and only
  if
  $$x\not\in\bigcup_{k<\omega} E_k.$$ In this proof, the game in the game 
  scheme $G_\I(Y,T)$ with the above payoff set will be also
  denoted by $G_\I(Y,T)$ (this should not cause confusion
  since we are not going to consider other payoff sets in
  the game scheme $G_\I(Y,T)$).

  \medskip

  Here is one more piece of notation.
  \begin{itemize}
  \item If $S$ is a subtree of the game scheme $G_\I(Y,T)$,
    then by $\hat S\subseteq T$ we denote the tree built
    from the moves of Adam in partial plays in $S$ (i.e. we
    forget about Eve's moves). We write $\proj[S]$ for
    $\proj[\hat S]$.
  \item If $Y'=Y\times Z$, $z\in Z$ is fixed and $T'$ is a
    tree on $Y'$ such that $T^z\subseteq T$, then by $S^z$
    we denote the subtree of $G_\I(Y',T')$, in which the
    moves $\tau$ of Adam are changed to $\tau^z$.
  \end{itemize}

  \begin{lemma}\label{unfoldgame}
    The game $G_\I(Y,T)$ is determined. Eve has a winning
    strategy in $G_\I(Y,T)$ if and only if
    $$\proj[T]\in \I.$$
  \end{lemma}
  \begin{proof}
    Suppose first that $\proj[T]\in \I$. Then Eve chooses
    $\emptyset$ in all her moves and wins the game.

    On the other hand, suppose that $\proj[T]$ is
    $\I$-positive. We define a winning strategy for Adam as
    follows. In his moves, Adam constructs $\tau_n\in T$ so
    that
    \begin{itemize}
    \item $[\bar\tau_n]\subseteq O_{n-1}$,
    \item $\proj[T(\tau_n)]\not\in \I$.
    \end{itemize}
    Suppose Adam is about to make his $n$-th move, his
    previous move is $\tau_{n-1}$ and the last move of Eve
    is $O_{n-1}$ ($O_{-1}=\emptyset$). Using the fact that
    $O_{n-1}\cap\proj[T(\tau_{n-1})]\not\in \I$, Adam picks
    $\tau_n\in T$ extending $\tau_{n-1}$ such that
    $[\bar\tau_n]\subseteq O_{n-1}$ and
    $\proj[T(\tau_n)]\not\in \I$. This is the strategy for
    Adam.  It is winning since after each play we have that
    $x\in O_n$ for each $n<\omega$, so in particular
    $x\in\bigcup_{k<\omega} E_k$.
  \end{proof}

  \begin{remark}\label{bigstrategy}
    Note that if $S$ is a winning strategy for Adam in the
    game $G_\I(Y,T)$, then for each partial play $\pi\in S$,
    we have $\proj[S(\pi)]\not\in \I$. This is because
    otherwise we could construct a counterplay to the
    strategy $S$. In particular, if $\tau$ is the last move
    of Adam in $\pi$, then we have $\proj[T(\tau)]\not\in
    \I$.
  \end{remark}

  Now we define the key notion in this proof. Let $\pi$ be a
  partial play in $G_\I(Y,T)$ of length $2l$, in which Eve
  chooses clopen sets $O_i$, for $i<l$, and Adam picks
  $\tau_{l-1}\in T$ in his last move.  Suppose that
  $Y'=Y\times Z$, $Z\in H_{\omega_1}$, $z\in Z$ is fixed and
  $T'$ is a tree on $\omega\times Y'$ such that
  $(\tau_{l-1})^z\in T'$.  By the \textit{relativized
    unfolded game} $G_\I(Y',T')^z_\pi$ we mean the game, in
  which
  \begin{itemize}
  \item in his $n$-th move Adam picks $\tau_{n+l}'\in T'$,
    $\tau_{n+l}'\supsetneq\tau_{n+l-1}'$
    ($\tau_{l-1}'=(\tau_{l-1})^z$),
  \item in her $n$-th move Eve picks a clopen set $O_{n+l}$
    in $\baire$ such that $$\proj[T'(\tau_{n+l}')]\not\in
    \I\ \Rightarrow\
    O_{n+l}\cap\proj[T'(\tau_{n+l}')]\not\in \I.$$
  \end{itemize}
  The payoff set is the same as in the unrelativized case,
  i.e. we use all $\langle O_n:n<\omega\rangle$ to define a
  sequence of closed sets $\langle E_k:k<\omega\rangle$, we
  put $x=\bigcup_{l-1\leq n<\omega}\bar{\tau_n'}$ and Eve wins if
  and only if $x\in\bigcup_{k<\omega} E_k$.

  With an analogous proof as in Lemma \ref{unfoldgame} we
  get the following lemma.

  \begin{lemma}\label{relunfoldgame}
    Suppose $\pi$ is a partial play in $G_\I(Y,T)$ and $\tau$
    is the last move of Adam in $\pi$. Let
    $G_\I(Y',T')^z_\pi$ be a relativized unfolded game. Eve
    has a winning strategy in $G_\I(Y',T')^z_\pi$ if and only
    if
    $$\proj [T'(\tau^z)]\in \I.$$
  \end{lemma}

  \begin{lemma}\label{pilim}
    If $S$ is a winning strategy for Adam in $G_\I(Y,T)$ then
    $$\proj[S]\not\in \I.$$
  \end{lemma}

  \begin{proof}
    Let $A=\proj[S]$. If $A\in \I$, then there are closed
    sets $E_k\in \I$ such that $A\subseteq\bigcup_{k<\omega}
    E_k$.  Let $U^m_k$ be clopen sets such that
    $U^m_k\subseteq U^{m+1}_k$ and $\baire\setminus
    E_k=\bigcup_{m<\omega} U^m_k$ for each $k<\omega$. We
    construct an Eve's counterplay to the strategy $S$ in
    the following way.  Suppose she is to make her $n$-th
    move and let $\tau_n$ be the last move of Adam. By
    Remark \ref{bigstrategy}, $\proj[T(\tau_n)]\not\in \I$.
    Let $\rho(n)=(i,k)$. She chooses $m\geq n$ big enough so
    that
    $$U^m_k\cap\proj[T(\tau_n)]\not\in \I.$$ Let her $n$-th move
    be $O_n=U^m_k$. If she plays in this way then
    $$\bigcup_{i<\omega}O_{\rho^{-1}(i,k)}=\baire\setminus
    E_k,$$ i.e. the closed sets she gets are precisely the
    sets $E_k$.  If $x=\bigcup_{n<\omega}\bar\tau_n$ is the
    point in $\baire$ constructed by Adam, then by the
    definition of $A$, $x\in A\subseteq\bigcup_{k<\omega}
    E_k$, which shows that Eve wins.
  \end{proof}

  \medskip

  Note that it follows from Lemmas \ref{unfoldgame} and
  \ref{pilim} that any analytic $\I$-positive set
  $A\subseteq\baire$ contains an analytic $\I$-positive
  subset of the form $\proj[S]$ for a winning strategy $S$
  for Adam in a game $G_\I(Y,T)$ (where $Y=\omega$ and $T$ is
  a tree on $\omega\times\omega$ such that $A=\proj[T]$).

  \medskip

  Let $\T_\I$ be the set of all triples $(Y,T,S)$ where
  $Y\in H_{\omega_1}$, $T$ is a tree on $\omega\times Y$ and
  $S$ is a winning strategy for Adam in the game
  $G_\I(Y,T)$.  $\T_\I$ is a forcing notion with the
  following ordering: for $(Y',T',S'),(Y,T,S)\in\T_\I$ let
  $$(Y',T',S')\leq(Y,T,S)\quad\mbox{iff}\quad\proj[S']\subseteq\proj[S].$$

  Notice that $(Y,T,S)\mapsto\proj[S]$ is a dense embedding
  from $\T_\I$ to $\Q_\I=(\ana\setminus \I,\subseteq)$.
  Indeed, suppose that $(Y',T',S')\perp(Y,T,S)$. If
  $\proj[S']$ and $\proj[S]$ were compatible in $\Q_\I$,
  then we would find an $\I$-positive $\ana$ set
  $A\subseteq\baire$ such that
  $A\subseteq\proj[S']\cap\proj[S]$. Take any tree $T$ on
  $\omega\times\omega$ such that $\proj[T]=A$ and find a
  winning strategy $S''$ for Adam in $G_\I(\omega,T)$. Then
  $(\omega,T,S'')\leq (Y',T',S'),(Y,T,S)$, a contradiction.

  By Theorem \ref{Sol.dich}, $\P_\I$ is dense in $\Q_\I$ .
  Therefore the three forcing notions $\T_\I$, $\Q_\I$ and
  $\P_\I$ are equivalent.  We will show that the forcing
  $\T_\I$ satisfies Axiom A.

  \medskip

  Take $(Y,T,S)\in\T_\I$ and recall that for each play
  $p\in\lim S$ ending with $t\in\baire\times Y^\omega$, with
  $x\in\baire$ (defined from the moves of Adam) and a
  sequence of closed sets $E_n$ (defined from the moves of
  Eve), we have $x\not\in\bigcup_k E_k$.  Note that for each
  $k\in\omega$ there is $n\in\omega$ (even) such that (the
  partial play) $t\har n$ already determines that $x\not\in
  E_k$ (i.e.  $[\bar\tau_n]\subseteq O_m$ for some
  $m<\omega$ such that $\rho(m)=(i,k)$ for some $i<\omega$).
  Let $n_0(p)\in\omega$ be the minimal such $n$ for $k=0$.
  Put
  $$F_0(S)=\{p\har n_0(p): p\in\lim S\}.$$ Note that $F_0(S)$ is a
  front in $S$. Analogously we define $F_k(S)$, for each
  $k<\omega$ (instead of $E_0$ take $E_k$ and put
  $n_k(p)>n_{k-1}(p)$ minimal even number such that $p\har
  n_k(p)$ determines $x\not\in E_k$).

  \medskip

  Define $(Y',T',S')\leq_k (Y,T,S)$ iff
  \begin{itemize}
  \item[(i)] $(Y',T',S')\leq (Y,T,S)$,
  \item[(ii)] there is $Z\in H_{\omega_1}$ such that
    $Y'=Y\times Z$,
  \item[(iii)] there is $z\in Z$ such that $T^z\subseteq
    T'$,
  \item[(iv)] $(T')_Y\subseteq\hat S$,
  \item[(v)] $F_k(S')= F_k(S)^z$.
  \end{itemize}
  
  \medskip

  We will prove that $\T_\I$ satisfies Axiom A with the
  inequalities $\leq_k$. The condition (ii) serves for
  unfolding the game and condition (iii) is later used to
  make the unfolding ``rigid''. Condition (iv) is a
  technical detail. The crucial one is (v), which says that
  the ``splitting levels'' are kept up to $k$-th in the
  $k$-th step of the fusion.

  \medskip

  \textbf{1.} Fix $k<\omega$. Suppose that $(Y,T,S)\in
  \T_\I$ and $\dot\alpha$ is a name for an ordinal. We shall
  find $(Y',T',S')\leq_k (Y,T,S)$ and a countable set of
  ordinals $A$ such that
  $(Y',T',S')\Vdash_{\T_\I}\dot\alpha\in\check A$.

  For each $\pi\in F_k(S)$ find an ordinal $\alpha_\pi$ and
  an $\I$-positive analytic set $A_\pi\subseteq\proj[S(\pi)]$
  (recall that $\proj[S(\pi)]$ is $\I$-positive by Remark
  \ref{bigstrategy}) such that
  $$A_\pi\Vdash_{\Q_\I}\dot\alpha=\check\alpha_\pi.$$ Let
  $\tau_\pi\in T$ be the last move of Adam in $\pi$. Next,
  pick $Z_\pi\in H_{\omega_1}$ such that $0\in Z_\pi$ ($0$
  will be used as $z$ from (iii); it does not matter what
  element we choose for $z$, as long as it belongs to
  $H_{\omega_1}$) and find a pruned tree $T_\pi$ on
  $\omega\times Y\times Z_\pi$ such that
  \begin{itemize}
  \item ${\tau_\pi}^0\in T_\pi$ and ${\tau_\pi}^0$ is a stem
    of $T_\pi$,
  \item $(T_\pi)_Y\subseteq \hat S$,
  \item $A_\pi=\proj[T_\pi]$
  \end{itemize}
  ($T_\pi$ is chosen such that its projection to
  $\baire\times Y^\omega$ is the analytic set $\lim\hat
  S\cap\p_Y^{-1}[A_\pi]$). Ensure also that
  \begin{itemize}
  \item for each $\pi,\pi'\in F_k(S)$, if $\pi\not=\pi'$,
    then $Z_\pi\cap Z_{\pi'}=\{0\}$,
  \item for each $\tau\in T_\pi$ if
    $\tau\supsetneq(\tau_\pi)^0$, then $\tau(|\tau|-1)\in
    \omega\times Y\times (Z_\pi\setminus\{0\})$,
  \end{itemize}
  Put
  $$Z=\bigcup_{\pi\in F_k(S)}Z_\pi,\quad z=0,\quad
  Y'=Y\times Z.$$ For each $\pi\in F_k(S)$ let $S_\pi$ be a
  winning strategy for Adam in $G_\I(Y\times Z_\pi,
  T_\pi)^0_\pi$. Such a strategy exists by Lemma
  \ref{relunfoldgame} since
  $\proj[T_\pi({\tau_\pi}^0)]=\proj[T_\pi]=A_\pi\not\in \I$
  (recall that ${\tau_\pi}^0$ is a stem of $T_\pi$).  Let
  $$T'=T^0\cup\bigcup_{\pi\in F_k(S)} T_\pi$$ and consider
  the game $G_\I(Y',T')$. The tree
  $$S'=\bigcup_{\pi\in F_n(S)}(\pi^0)^\smallfrown S_\pi$$ is
  a strategy in $G(Y',T')$ since all $\pi^0$, for $\pi\in
  F_k(S)$, are partial plays in $G(Y',T')$ (because
  $T^0\subseteq T'$).  Moreover, it is a winning strategy
  for Adam since each $S_\pi$ is a winning strategy for Adam
  in $G_\I(Y\times Y_\pi, T_\pi)_\pi$.  Therefore
  $(Y',T',S')\in \T_\I$. By the construction we have
  $(Y',T',S')\leq_k(Y,T,S)$. Moreover,
  $$(Y',T',S')\Vdash_{\T_\I}\dot\alpha\in\{\alpha_\tau: \tau\in
  F_k(S)\}$$ because the set $\{\proj[S'(\pi)]: \pi\in
  F_k(S)\}$ is predense below $\proj[S']$ and we have
  $\proj[S'(\pi)]\Vdash_{\Q_\I}\dot\alpha=\check\alpha_\pi$
  (since $\proj[S'(\pi)]\subseteq A_\pi$).

  \medskip

  \textbf{2.} Let $\langle(Y_k,T_k,S_k): k<\omega\rangle$ be
  a fusion sequence. For each $k<\omega$ let $Z_k\in
  H_{\omega_1}$ be such that $Y_{k+1}=Y_k\times Z_k$ and let
  $z_k\in Z_k$ be as in the definition of $\leq_k$. Let
  $\vec z_k=\langle z_k, z_{k+1}, \ldots\rangle$. Put
  $Y=\bigcup_{k<\omega}(Y_k)^{\vec z_k}$ and
  $T=\bigcup_{k<\omega} (T_k)^{\vec z_k}$. $T$ is a tree on
  $Y$. Notice that for each $k<\omega$, for each $\tau\in
  T_k$ we have
  \begin{displaymath}
    \proj[T_k(\tau)]=\proj[T(\tau^{\vec
      z_k})].\tag{$*$} 
  \end{displaymath}
  Indeed, $\proj[T_k(\tau)]\subseteq\proj[T(\tau^{\vec
    z_k})]$ follows from (iii) and $\proj[T(\tau^{\vec
    z_k})]\subseteq\proj[T_k(\tau)]$ from (iv) (because for
  each $t\in\lim T(\tau^{\vec z_k})$ a sequence of its
  initial coordinates is in $\lim \hat S_k$ and hence in
  $\lim T_k$)

  Consider the game $G_\I(Y,T)$ and let
  $$S=\bigcup_{k<\omega}
  F_k(S_k)^{\vec z_k}.$$ Note that it follows from ($*$)
  that $S$ is a strategy for Adam in $G_\I(Y,T)$.  Moreover,
  for each $k<\omega$ we have $F_k(S)= F_k(S_k)^{\vec z_k}$,
  by the definition of $F_k$. Since for each $p\in\lim S$
  we have
  $$\forall k<\omega\ \exists m<\omega\quad p\har m\in
  F_k(S),$$ it follows that $S$ is a winning strategy for
  Adam in $G_\I(Y,T)$. Therefore $(Y,T,S)\in\T_\I$.

  To see that $(Y,T,S)\leq (Y_k,T_k,S_k)$ we use the
  property (iv). Indeed, if $x\in\proj[S]$, then there is a
  play in $\lim S$, in which $x$ is defined. By (iv),
  however, we can extract from this play a play in $\lim
  S_k$, in which $x$ is defined.

  To check that $(Y,T,S)\leq_k (Y_k,T_k,S_k)$ we put
  $Z=\prod_{m\geq k}Z_m$ and $z=\vec z_k$.

  \medskip

  This ends the proof of Theorem \ref{axiomaclosed}.
\end{proof}

\section{Coanalytic families of closed
  sets}\label{sec:coanaidclosed}

If $\K$ is a family of closed subsets of a Polish space $X$,
then its projective complexity can be defined in terms of
the Effros space $F(X)$. Namely, if $\class$ is a projective
pointclass, then we say that $\K$ \textit{is $\class$} if it
belongs to $\class$ in $F(X)$.

Recall that if $X=\baire$, then for each closed set
$C\subseteq X$ there is a pruned subtree $T$ of
$\omega^{<\omega}$ such that $C=\lim T$. If $X$ is an
arbitrary Polish space, then for each closed set $C\subseteq
X$ the family $\mathcal{U}=\{U\mbox{ basic open}: U\cap
C=\emptyset)\}$ can be treated as a code for $C$ (since
$C=X\setminus\bigcup\mathcal{U}$). Moreover, the family
$\mathcal{U}$ has the following property
\begin{displaymath}
  \forall U \mbox{ basic open}\quad U\subseteq\bigcup\mathcal{U}\
  \Rightarrow\ U\in\mathcal{U}\tag{$*$}
\end{displaymath}
We can code all families $\mathcal{U}$ satisfying $(*)$ by
elements of $\baire$ and create a universal closed set
$\tilde C\subseteq\baire\times X$ such that if $t\in\baire$
codes $\mathcal{U}(t)$, then $\tilde
C_t=X\setminus\bigcup\mathcal{U}(t)$.

Using the property $(*)$ of the coding, we can check that
the function $$\baire\ni t\mapsto \tilde C_t\in F(X)$$ is
Borel measurable (i.e.  preimages of Borel sets in $F(X)$
are Borel). Therefore, for any projective pointclass
$\class$, a family of closed sets $\K$ is $\class$ if and
only if the set $\{t\in\baire: \tilde C_t\in \K\}$ is
$\class$ in $\baire$.

\medskip

The projective complexity of families of closed subsets of
$X$ can be also generalized to families of sets in other
Borel pointclasses --- in terms of universal sets. We will
now introduce a Borel structure on the family of $\gdelta$
sets. 

Note that for any $\gdelta$ set $G\subseteq\baire$ there is
a pruned tree $T\subseteq\otr$ and a family
$\langle\sigma_\tau\in\otr:\ \tau\in T\rangle$ such that the
family of clopen sets $\langle[\sigma_\tau]:\ \tau\in
T\rangle$ forms a Lusin scheme and
$G=\bigcap_{n<\omega}\bigcup_{\tau\in T\cap
  \omega^n}[\sigma_\tau]$.  Generalizing this to an
arbitrary Polish space $X$ we claim that for any $\gdelta$
set $G$ in $X$ there is a Souslin scheme $\langle
U_\tau:\tau\in\otr\rangle$ of basic open sets such that

\begin{itemize}
\item[(i)] $\diam(U_\tau)<1\slash |\tau|$,
\item[(ii)] $\cl{U_\tau}\subseteq
  U_{\tau\upharpoonright(|\tau|-1)}$
\item[(iii)] if $U_\tau\not=\emptyset$ then $U_{\tau^\smallfrown
    n}\not=\emptyset$ for some $n<\omega$
\end{itemize}
and $G=\bigcap_{n<\omega}\bigcup_{|\tau|=n} U_\tau$. Indeed,
if $G=\bigcap_{n<\omega} O_n$ (each $O_n$ open and
$O_{n+1}\subseteq O_n$), then we construct a Souslin scheme
$U_\tau$ by induction on $|\tau|$ as follows. Having all
$U_\tau$ for $|\tau|\leq n$ we find a family
$\{U_\tau:\tau\in\omega^{n+1}\}$ such that
\begin{itemize}
\item for each $\tau\in\omega^{n+1}$ we have $U_\tau\cap
  G\not=\emptyset$,
\item for each $\sigma\in\omega^n$ we have $U_\sigma\cap
  O_{n+1}=\bigcup\{\cl{U_\tau}: \sigma\subseteq\tau,
  \tau\in\omega^{n+1}\}$.
\end{itemize} 
Let us code all Souslin schemes of clopen sets satisfying
(i)--(iii) by elements of the Baire space $\baire$ and
create a universal $\gdelta$ set $\tilde
G\subseteq\baire\times X$ such that if $t\in\baire$ codes a
Souslin scheme $\langle U_\tau(t):\tau\in\otr\rangle$, then
$\tilde G_t=\bigcap_{n<\omega}\bigcup_{|\tau|=n} U_\tau(t)$.

\begin{remark}
  Here we show how the above coding is done in case of the
  Baire space and Luzin schemes (the general case is
  analogous). We pick any bijection $\rho$ between $\omega$
  and $\btree$ and consider the set $H$ of all elements of
  $\baire$ which code (via $\rho$) a Luzin scheme satisfying
  (iii). This set is a $\gdelta$ set and thus there is a
  continuous bijection $f:\baire\rightarrow H$. Now, we say
  that $x\in\baire$ codes a $\gdelta$ set $G$, if $f(x)$
  codes a Luzin scheme $U_\tau$ such that
  $G=\bigcap_{n<\omega}\bigcup_{|\tau|=n}U_\tau$.
\end{remark}

\begin{lemma}\label{gdeltaeffros}
  If $U\subseteq X$ is open, then
  $$\{t\in\baire:\ \tilde G_t\cap U\not=\emptyset\}\mbox{ is open}.$$
\end{lemma}
\begin{proof}
  Note that by (iii) and (ii), $\tilde G_t\cap
  U\not=\emptyset$ if and only if there is a nonempty basic
  open set $V\subseteq U$ such that $V$ occurs in the
  Souslin scheme coded by $t$.
\end{proof}

If $\class$ is a projective pointclass and $\G$ is a family
of $\gdelta$ sets, then we say that $\G$\textit{ is
  $\class$} if $\{t\in\baire: \tilde G_t\in\G\}$ is $\class$
in $\baire$. By Lemma \ref{gdeltaeffros} the map
$$\gdelta\ni G\mapsto\cl{G}\in F(X)$$ is Borel
(i.e. preimages of Borel sets in $F(X)$ are Borel).

\medskip

Let $\K$ be a family of closed sets in a Polish space $X$.
We say that $\K$ is \textit{hereditary} if for any two
closed sets $C,D$ such that $C\subseteq D$, if $D\in\K$,
then $C\in\K$.

Let $\I$ be a $\sigma$-ideal on a Polish space $X$ and
$A\subseteq X$. We say that $A$ is \textit{$\I$-perfect} if
$A\not=\emptyset$ and for each open set $U$ the set $A\cap
U$ is either empty or $\I$-positive. If $\K$ is a family of
closed sets in a Polish space $X$ and $D\subseteq X$ is
closed, then we say that $D$ is \textit{$\K$-perfect} if the
sets from $\K$ have relatively empty interior on $D$.  Note
that if $\K$ is hereditary, then a closed set $D$ is
$\K$-perfect if and only if for each basic open set $U$ in
$X$, either $U\cap D=\emptyset$, or else $\cl{U\cap
  D}\not\in\K$.

\begin{lemma}\label{iperfect}
  Let $\I$ be a $\sigma$-ideal generated by closed sets on a
  Polish space $X$. If $G\subseteq\baire$ is a $\gdelta$ set
  and $\cl{G}$ is $\I$-perfect, then $G\not\in\I$.
\end{lemma}
\begin{proof}
  Let $C=\cl{G}$ and suppose $C$ is $\I$-perfect yet
  $G\in\I$. If $G\subseteq \bigcup_n F_n$ and $F_n$ are
  closed sets in $\I$, then each $F_n\cap C$ is a closed
  nowhere dense subset of $C$. This contradicts the Baire
  category theorem.
\end{proof}

\begin{lemma}\label{iperfectdense}
  Let $\I$ be a $\sigma$-ideal generated by closed sets on a
  Polish space $X$. If $G\subseteq\baire$ is an
  $\I$-positive $\gdelta$ set, then it contains an
  $\I$-perfect $\gdelta$ set $G'$.
\end{lemma}

\begin{proof}
  Put $G'=G\setminus\bigcup\{U: U\mbox{ is basic open set
    and }U\cap G\in\I\}$.
\end{proof}

\begin{lemma}\label{iperfectcl} Let $X$ be a Polish space.
  \begin{enumerate}
  \item[(i)] Let $\I$ be a $\sigma$-ideal on $X$ generated
    by closed sets. If $G\subseteq X$ is a $\gdelta$ set,
    then $G$ is $\I$-perfect if and only if $\cl{G}$ is
    $\I$-perfect.
  \item[(ii)] Let $\K$ be a family of closed subsets of $X$,
    let $\sigma(\K)$ be the $\sigma$-ideal of closed sets
    generated by $\K$ and let $\I$ be the $\sigma$-ideal
    generated by $\K$. If $D\subseteq X$ is closed, then the
    following are equivalent
    \begin{itemize}
    \item $D$ is $\K$-perfect,
    \item $D$ is $\sigma(\K)$-perfect,
    \item $D$ is $\I$-perfect.
    \end{itemize}
  \end{enumerate}
\end{lemma}
\begin{proof}
  (i) Clearly, if $G$ is $\I$-perfect, then $\cl{G}$ is also
  $\I$-perfect. Suppose $\cl{G}$ is $\I$-perfect but $G$ is
  not $\I$-perfect. Then we can find an open set $U$ such
  that $U\cap G\in\I$ and $U\cap G\not=\emptyset$. Consider
  $U\cap\cl{G}$, which is $\I$-perfect because $\cl{G}$ is
  $\I$-perfect. $U\cap\cl{G}$ is a Polish $\I$-perfect space
  which contains a dense $\gdelta$ set in $\I$. By Lemma
  \ref{iperfect} we get a contradiction with the Baire
  category theorem.

  (ii) This follows directly from the Baire category theorem.
\end{proof}

\begin{lemma}\label{iperfectgdelta}
  Let $X$ be a Polish space and let $\I$ be a $\coana$ on
  $\ana$ $\sigma$-ideal generated by closed sets on $X$. Let
  $\K$ be a hereditary coanalytic family of closed sets on
  $X$. Then
  \begin{enumerate}
  \item[(i)] the family of $\I$-perfect $\gdelta$ sets is
    $\ana$,
  \item[(ii)] the family of $\K$-perfect closed sets is
    $\ana$,
  \item[(iii)] the family of $\gdelta$ sets with
    $\K$-perfect closure is $\ana$.
  \end{enumerate}
\end{lemma}
\begin{proof}
  (i) We see that $G\in\gdelta$ is $\I$-perfect if and only
  if
  $$G\not=\emptyset\quad\wedge\quad\forall
  U\mbox{ basic open } (G\cap U\not=\emptyset\Rightarrow
  G\cap U\not\in\I).$$ This is a $\ana$ condition by Lemma
  \ref{gdeltaeffros} and the assumption that $\I$ is
  $\coana$ on $\ana$.

  (ii) Note that a closed set $D$ is $\K$-perfect if and
  only if
  $$D\not=\emptyset\quad\wedge\quad\forall U\mbox{ basic
    open }\ (D\cap U\not=\emptyset\Rightarrow \cl{D\cap
    U}\not\in \K).$$ This is $\ana$ condition since the
  closure is a Borel map.

  (iii) This follows (ii) and the fact that the closure is a
  Borel map.
\end{proof}

\begin{remark}
  In addition to coding $\gdelta$ sets by Luzin schemes, we
  can also code continuous partial functions on their dense
  $\gdelta$ subsets, that is triples $(G,f,G')$ where $G$ is
  a $\gdelta$ set, $G'$ is a dense $\gdelta$ subset of $G$
  and $f:G'\rightarrow\baire$ is continuous. For a sample
  method of coding see e.g.  \cite[Proposition
  2.6]{Kechris}. If $D\subseteq\baire\times\baire$ is a
  closed set and $G\subseteq\baire$ is a $\gdelta$ set, then
  we write
  $$f:G\xrightarrow{*} D$$ to denote that $f$
  is a continuous function from a dense $\gdelta$ subset of
  $G$ and the graph of $f$ is contained in $D$.
\end{remark}

It is well-known (see \cite[Theorem 35.38]{Kechris} or
\cite[Lemma 4.8]{Far.Zpl}) that if $\K$ is a coanalytic
hereditary family of closed sets, then the $\sigma$-ideal
generated by $\K$ is $\coana$ on $\ana$. Let us present a
new proof of this fact, which uses idealized forcing and
Theorem \ref{Sol.dich}.

\begin{corollary}\label{coanaonana}
  Let $X$ be a Polish space. If $\K$ is a coanalytic
  hereditary family of closed sets in $X$, then the
  $\sigma$-ideal generated by $\K$ is $\coana$ on $\ana$.
\end{corollary}
\begin{proof}
  Let $\I$ be the $\sigma$-ideal generated by $\K$ and let
  $A\subseteq X\times\ X$ be $\ana$. Denote by $\G$ the
  family of $\I$-perfect $\gdelta$ sets. By Lemmas
  \ref{iperfectcl} and \ref{iperfectgdelta}, $\G$ is $\ana$.
  By Theorem \ref{Sol.dich} and Lemma \ref{iperfectdense},
  if $x\in X$, then
  $$A_x\not\in\I\qquad\mbox{iff}\qquad \exists G\in\G\quad
  G\subseteq A_x.$$ Let $D\subseteq X^2\times\baire$ be a
  closed set such that $A=\pi[D]$ ($\pi$ denotes the
  projection to the first two coordinates). Note that
  $G\subseteq A_x$ is equivalent to $$\forall y\in G\
  \exists z\in\baire\ (y,z)\in D_x.$$ By
  $\shoenfield$-absoluteness we get a name $\dot z_x$ such
  that $$G\Vdash (\dot g,\dot z_x)\in D_x.$$ Now, by
  continuous reading of names and properness of $\P_\I$ we
  get a $G'\in\G$, $G'\subseteq G$ and a continuous function
  $f_x:G'\rightarrow \baire$ reading $\dot z_x$. Notice that
  the graph of $f_x$ is contained in $D_x$, so
  $f_x:G'\xrightarrow{*} D_x$.  Coversely, if there is such
  function $f$, then $\dom(f)$ is an $\I$-perfect
  $\gdelta$-set contained in $A_x$. Thus, we have shown that
  $$A_x\not\in\I\qquad\mbox{iff}\qquad \exists G\in\G\quad
  \exists f:G\xrightarrow{*} D_x.$$ Using the coding of
  $\gdelta$ sets and partial continuous functions, one can
  easily check that $$\exists f:G\xrightarrow{*} D_x$$ is a
  $\ana$ formula. Thus, the whole formula is $\ana$ and we
  are done.
\end{proof}

Analytic sets in a Polish space $X$ can be coded by a
$\ana$-universal $\lana$ set on $\baire\times X$. In the
remaining part of this Section we fix a universal analytic
set $\tilde A\subseteq\baire\times X$ which is $\lana$ and
\textit{good} (cf. \cite[Section 3.H.1]{Mosch}). The set
$\tilde A$ will be used to code analytic sets in $X$ as well
as $\lana(t)$ sets for each $t\in\baire$.

We say that a set $S\subseteq\baire$ \textit{codes a
  $\sigma$-ideal $\I$ of analytic sets} if the family
$\I=\{\tilde A_t: t\in S\}$ is a $\sigma$-ideal of analytic
sets.  Let $\iform(v)$ be a $\coana$ formula. Note that the
famlily of analytic sets whose codes satisfy $\iform(v)$ is
$\coana$ on $\ana$ (because $\tilde A$ is good and
$\iform(v)$ is a $\coana$ formula).

\begin{lemma}[Folklore]\label{trik}
  If $\A$ is a $\coana$ on $\ana$ family of analytic sets,
  then $\A$ is downward closed, i.e. if $A,B\in\ana$ are
  such that $A\subseteq B$ and $B\in\A$, then $A\in\A$.
\end{lemma}
\begin{proof}
  Suppose $A\subseteq B$ are $\ana$ and $B\in\A$. Let
  $Z\subseteq\baire$ be such that $Z\in\ana\setminus\coana$.
  Take $L\subseteq\baire\times X$ such that $$(t,x)\in
  L\quad\Leftrightarrow\quad (t\in Z\wedge x\in B)\vee x\in
  A.$$ As $\{t\in\baire: L_t\in\A\}\in\coana$, we conclude
  that $B\in\A$.
\end{proof}

Suppose $V\subseteq W$ is a generic extension and in $V$ we
have a $\coana$ on $\ana$ $\sigma$-ideal $\I$. Let
$\iform(v)$ be a $\coana$ formula which codes the
$\sigma$-ideal of analytic sets $\I\cap\ana$. By $\I^W$ we
denote the family of analytic sets whose codes satisfy
$\iform(v)$ in $V[G]$. This definition does not depend on
the formula $\iform(v)$ since if $\iform'(v)$ is another
such formula, then
$$\forall t\in\baire\quad \iform(t)\Leftrightarrow \iform'(t)$$ is a
$\negshoenfield$ sentence and hence it is absolute for
$V\subseteq W$.

Now we prove Theorem \ref{absoluteness}

\begin{proof}[Proof of Theorem \ref{absoluteness}]
  Let $\kform(v)$ be a $\coana$ formula defining the set of
  codes of closed sets in $\I$. By $\K^W$ we denote the
  family of closed sets in $W$, whose codes satisfy
  $\kform(v)$ (as previously, this does not depend on the
  formula $\kform(v)$).

  First we show that in $W$ the family $\K^W$ is hereditary.
  Consider the following sentence $$\exists
  t,s\in\baire\quad (\neg\kform(s)\ \wedge\ \kform(t)\
  \wedge\ \tilde C_s\subseteq\tilde C_t).$$ It is routine to
  check that it is $\shoenfield$ and hence absolute for
  $V\subseteq W$. This shows that $\K^W$ is hereditary.

  Next we show that $\I^W$ is a $\sigma$-ideal of analytic
  sets. Let $D\subseteq \baire\times X\times\baire$ be a
  closed set such that $\pi[D]=\tilde A$ (here $\pi$ denotes
  the projection to the first two coordinates).  Consider
  the following formula $\iform'(v)$:
  $$\neg\ (\exists G\in\gdelta\quad \cl{G}\mbox{ is
  }\kform\mbox{-perfect}\ \wedge\ \exists f:G\xrightarrow{*}
  D_v)$$ (writing that $\cl{G}$ is $\kform$-perfect we mean
  that $\cl{G}$ is perfect with respect to the family of
  closed sets defined by $\kform(v)$).  Using Lemma
  \ref{iperfectgdelta}(iii) we can check that $\iform'(v)$
  is a $\coana$ formula. From the proof of Corollary
  \ref{coanaonana} and from Lemmas \ref{iperfect},
  \ref{iperfectdense} and \ref{iperfectcl} we conclude that
  in $V$ we have $$\forall t\in \baire\quad
  \iform(t)\Leftrightarrow \iform'(t).$$ This is a
  $\negshoenfield$ sentence and hence it holds in $W$.
  Therefore, it is enough to check that
  $(\iform')^W(\baire)$ codes a $\sigma$-ideal of analytic
  sets. However, it follows from Theorem \ref{Sol.dich} and
  from Lemmas \ref{iperfect}, \ref{iperfectdense} and
  \ref{iperfectcl}, that $(\iform')^W(\baire)$ codes the
  $\sigma$-ideal generated by $\kform(\baire)^W$.

  The fact that $\I^W$ is $\coana$ on $\ana$ follows now
  from the remarks preceding this proposition.
\end{proof}

We also have the following alternative proof.

\begin{proof}[Alternative proof of Theorem
  \ref{absoluteness}] Throughout this proof we denote the
  closure of a set $A$ by $\ocl{A}$. Withoug loss of
  generality assume that $\I$ is $\lcoana$ on $\lana$ and
  $X=\baire$. 

  We will use the following notation. If $\varphi(v)$ is a
  formula and $t\in\baire$, then by $\lana(t)\wedge\varphi$
  (respectively $\lbor(t)\wedge\varphi$) we denote the
  family of $\lana(t)$ (respectively $\lbor(t)$) sets whose
  codes satisfy $\varphi(v)$.

  Let $\iform(v)$ be a $\lcoana$ formula defining the set of
  codes of analytic sets in $\I$. Let $\K$ be the family of
  closed sets in $\I$ and let $\kform(v)$ be a $\lcoana$
  formula defining the set of codes of the (closed) sets in
  $\K$ (in terms of the universal closed set $\tilde C$).
  Consider the formula $\hat K(v)$ saying that $\ocl{\tilde
    A_v}\in\K$.  Note that $\hat K(v)$ can be written as
  follows
  $$\forall s\in\baire\quad \tilde C_s\subseteq\ocl{\tilde
    A_v}\ \ \Rightarrow\ \ K(s)$$ and notice that it is a
  $\lcoana$ formula.

  Consider the set $F\subseteq\baire\times X$ defined as
  follows: $$(t,x)\in F\quad\mbox{ iff }\quad
  x\in\bigcup(\lbor(t)\wedge\hat K).$$ By the usual coding
  of $\lbor$ sets we get that $F$ is $\lcoana$.
  \begin{lemma}\label{effective}
    For each $t\in\baire$ we have
    $$\bigcup(\lana(t)\wedge\hat K)=\bigcup(\lbor(t)\wedge
    \hat K)=\bigcup(\lana(t)\wedge I).$$
  \end{lemma}
  \begin{proof}
    Without loss of generality assume that $t=0$. The first
    equalitiy follows from the First Reflection Theorem
    (since $\hat K(v)$ is a $\lcoana$ formula). Denote
    $C=\bigcup(\lbor(t)\wedge \hat K)=\bigcup(\lana(t)\wedge
    \hat K)$.

    In the second equality, the left-to-right inclusion is
    obvious since $\hat \kform(s)$ implies $\iform(s)$, for
    each $s\in\baire$. We need to prove that if $A\in\lana$
    is not contained in $C$, then $A\not\in\I$. Suppose
    $A\in\lana$ and $A\not\subseteq C$. Since $C\in\lcoana$,
    we may assume that $A\cap C=\emptyset$. Let $T$ be a
    recursive pruned tree on $\omega\times\omega$ such that
    $A=\proj[T]$. If $A\in\I$, then there is a sequence of
    closed sets $\langle D_n:n<\omega\rangle$ such that each
    $D_n\in\I$ and $A\subseteq\bigcup_{n<\omega} D_n$. By
    induction we construct a sequence of
    $\langle\tau_n\in\otr\rangle$ and $\sigma_n\in T$ such
    that for each $n<\omega$ the following hold
    \begin{itemize}
    \item $\sigma_{n+1}\supsetneq\sigma_n$ and
      $\tau_{n+1}\supsetneq\tau_n$,
    \item $\proj[T(\sigma_n)]\subseteq [\tau_n]$,
    \item $\proj[T(\sigma_{n-1})]\cap[\tau_n]\,\cap\,
      D_n=\emptyset$
    \end{itemize}

    We take $\sigma_{-1}=\emptyset$. Suppose $\sigma_n$ and
    $\tau_n$ are constructed. Notice that
    $\proj[T(\sigma_n)]$ is $\lana$. Since $A\cap
    C=\emptyset$ we see that
    $\ocl{(\proj[T(\sigma_n)])}\not\in\K$.  Consequently,
    $\proj[T(\sigma_n)]\not\subseteq D_n$ and hence there is
    $\tau_{n+1}\supsetneq\tau_n$,
    $[\tau_{n+1}]\subseteq\tau_n$ such that
    \begin{itemize}
    \item[(i)]
      $\proj[T(\sigma_n)]\cap[\tau_{n+1}]\not=\emptyset$,
    \item[(ii)] $\proj[T(\sigma_n)]\cap\cap[\tau_{n+1}]\cap
      D_n=\emptyset$.
    \end{itemize}
    Using (i) find $\sigma_{n+1}\supsetneq\sigma_n$ such
    that $\sigma_{n+1}\in T$ and
    $\proj[T(\sigma_{n+1})]\subseteq[\tau_{n+1}]$.

    Now, if $s=\bigcup_{n<\omega}\sigma_n$, then $s\in\lim
    T$, so $\pi(s)\in A$, but
    $\pi(s)\not\in\bigcup_{n<\omega}D_n$. This ends the
    proof of the lemma.
  \end{proof}

  Consider the following formula $\iform'(v)$ ($v$ is a
  variable):
  $$\forall z\in X\quad z\in\tilde A_v\Rightarrow z\in F_v.$$ Note
  that $\iform'$ is a $\lcoana$ formula and
  $$V\models\forall t\in\baire\quad \iform(t)\Leftrightarrow
  \iform'(t).$$ This is a $\lnegshoenfield$ sentence, so by
  absoluteness we see that $\iform$ and $\iform'$ define the
  same set of codes of analytic sets in $W$.

  \medskip

  Now we will show that $\I^W$ is a $\sigma$-ideal generated
  by closed sets. The fact that $\I^W$ is closed under
  taking analytic subsets follows from Lemma \ref{trik}
  because $\I^W$ is $\coana$ on $\ana$.

  Let us show that $\I^W$ is closed under countable unions.
  Pick a recursive bijection
  $\lceil\cdot\rceil:(\baire)^\omega\rightarrow \baire$. The
  following sentence
  \begin{displaymath}
    \forall \langle t_n: n<\omega\rangle\in
    (\baire)^\omega\quad ((\forall n<\omega\ \iform'(t_n))\
    \Rightarrow\ \iform'(\lceil t_n: n<\omega\rceil))\tag{$*$}
  \end{displaymath}
  is $\lnegshoenfield$ and hence it is absolute. Note that
  for any $\langle t_n: n<\omega\rangle\in(\baire)^\omega$
  we have $F_{t_k}\subseteq F_{\lceil t_n:n<\omega\rceil}$
  for each $k<\omega$ (because $t_k\in\lbor(\lceil
  t_n:n<\omega\rceil)$). Therefore $(*)$ holds in $V$ and
  hence also in $W$. This shows that the family of analytic
  sets coded by $(\iform')^{W}(\baire)$ is closed under
  countable unions.

  To see that $\I^W$ is generated by closed sets, take any
  $t\in W\cap\baire$ such that $({\tilde A}^W)_t\in\I^W$.
  This means that $W\models I'(t)$, so $({\tilde
    A}^W)_t\subseteq (F^W)_t$. Let $\langle
  t_n:n<\omega\rangle\in W$ be the the sequence of all
  elements of $\baire$ in $W$ which are $\lbor(t)$ and
  satisfy $\hat K(v)$. By the definition of $F$ we see that
  $$W\models({\tilde A}^W)_t\subseteq \bigcup_{n<\omega}({\tilde
    A}^W)_{t_n}$$ is satisfied in $W$. Let $\langle
  s_n:n<\omega\rangle\in W$ be a sequence of elements of
  $W\cap\baire$ such that $W\models({\tilde
    C}^W)_{s_n}=\ocl{({\tilde A}^W)_{t_n}}$ for each
  $n<\omega$. Now $W\models \hat K(t_n)$ implies $W\models
  K(s_n)$. Therefore $W\models({\tilde C}^W)_{s_n}\in\I^W$
  because $$\forall t,s\in\baire\quad(\tilde A_s\subseteq
  \tilde C_t\ \wedge\ K(t))\ \Rightarrow\ I(s)$$ is
  $\negshoenfield$ and holds in $V$. Since
  $$W\models({\tilde A}^W)_t\subseteq
  \bigcup_{n<\omega}({\tilde C}^W)_{s_n},$$ we conclude that
  $\I^W$ is generated by closed sets.
\end{proof}

\begin{remark}
  It is worth noting that analogously as in the alternative
  proof of Proposition \ref{absoluteness} we can get the
  following. If $\I$ is a $\coana$ on $\ana$ $\sigma$-ideal
  and $V\subseteq W$ is a generic extension, then $\I^W$ is
  a $\coana$ on $\ana$ $\sigma$-ideal in $W$.
\end{remark}

Zapletal defines in \cite{Zpl:FI} the class of
\textit{iterable} $\sigma$-ideals (see \cite[Definition
5.1.3]{Zpl:FI} for a definition without large cardinals, and
\cite[Definition 5.1.2]{Zpl:FI} for a definition under large
cardinal assumptions). $\coana$ on $\ana$ $\sigma$-ideals
generated by closed sets are iterable in the sense of
\cite[Definition 5.1.3]{Zpl:FI}.

\section{Products and iterations}\label{sec:productsanditerations}

If $\I$ is a $\sigma$-ideal on $X$, then we write
$\forall^\I x\in X\ \varphi(x)$ to denote that $\{x\in X:
\neg\varphi(x)\}\in\I$.  Let $\I$ and $\J$ be
$\sigma$-ideals on Polish spaces $X$ and $Y$, respectively.
Recall that the \textit{Fubini product of $\I$ and $\J$},
denoted by $\I\otimes\J$, is the $\sigma$-ideal of those
$A\subseteq X\times Y$ such that $$\forall^\I x\in X\
\forall^\J y\in Y\ (x,y)\not\in A.$$ If $\I_k$ is a
$\sigma$-ideal on $X_k$, for each $k<n$, then we naturally
extend the above definition to define $\bigotimes_{k<n}
\I_k=\I_0\otimes\bigotimes_{0<k<n}\I_k$. For each $n<\omega$
we also define the \textit{Fubini powers} of a
$\sigma$-ideal $\I$ as follows $\I^n=\bigotimes_{k<n}\I$.

\begin{lemma}[Folklore]\label{i2coanaonana}
  Suppose $\I$ and $\J$ are $\coana$ on $\ana$
  $\sigma$-ideals on Polish spaces $X$ and $Y$,
  respectively. Let $A\subseteq X\times Y$ be a $\ana$ set
  in $\I\otimes\J$. There is a $\ana$ set $D$ such that $A\cap
  D=\emptyset$ and $$\forall^\I x\in X\,\forall^\J y\in
  Y\quad (x,y)\in D.$$
\end{lemma}

\begin{proof}
  To simplify notation suppose that $X=Y=\baire$,
  $A\in\lana$, and $\I$ and $\J$ are $\lcoana$ on $\lana$.
  Put $$U^1=\bigcup(\lana\cap\I)$$ and let
  $U^2\subseteq\baire\times\baire$ be such that for each
  $t\in\baire$ we have $$(U^2)_t=\bigcup(\lana(t)\cap\J).$$
  By the First Reflection Theorem we have
  $U^1=\bigcup(\lbor\cap\I)$ and
  $(U^2)_t=\bigcup(\lbor(t)\cap\J)$ for each $t\in\baire$.
  Therefore, by the usual coding of $\lbor$ sets, we get
  that $U^1$ and $U^2$ are $\lcoana$. Put
  $$C=(U^1\times Y)\,\cup\, U^2.$$ Notice that $A\subseteq
  C$ (since otherwise we get that $A\not\in\I\otimes\J$).
  Now $B=X\times Y\setminus C$ is as needed.
\end{proof}

Generalizing the finite Fubini products, one can define the
Fubini product of length $\alpha$ for any $\alpha<\omega_1$.
A game-theoretic definition of $\I^\alpha$ is given in
\cite[Definition 5.1.1]{Zpl:FI}.  Definition \ref{Fubini}
below (equivalent to \cite[Definition 5.1.1]{Zpl:FI})
appears in \cite[p.  74]{CPA}. If $0<\beta<\alpha$ are
countable ordinals, then we write $\pi_{\alpha,\beta}$ for
the projection to the first $\beta$ coordinates from
$\prod_{\gamma<\alpha}X_\gamma$ to $\prod_{\gamma<\beta}
X_\gamma$. For each $D\subseteq
\prod_{\gamma<\alpha}X_\gamma$, we define
$\pi_{\alpha,0}[D]$ to be $X$. If $A\subseteq
\prod_{\gamma<\beta+1}X_\gamma$ and $x\in
\prod_{\gamma<\beta}X_\gamma$, then $A_x$ denotes the
vertical section of $A$ at $x$. If $A\subseteq X$ and $x\in
X$, then we put $A_x=A$.

\begin{definition}
  Let $\alpha$ be a countable ordinal, $\langle
  X_\beta:\beta<\alpha\rangle$ be a sequence of Polish spaces and
  $\bar\I=\langle \I_\beta:\beta<\alpha\rangle$ be a
  sequence of $\sigma$-ideals, $\I_\beta$ on $X_\beta$,
  respectively. We say that a set $D\subseteq
  \prod_{\beta<\alpha} X_\beta$ is an
  \textit{$\bar\I$-positive cube} if
  \begin{itemize}
  \item[(i)] for each $\beta<\alpha$ and for each
    $x\in\pi_{\alpha,\beta}[D]$ the set
    $$(\pi_{\alpha,\beta+1}[D])_x\mbox{ is }\I_{\beta+1}\mbox{-positive},$$
  \item[(ii)] for each limit $\beta<\alpha$ and $x\in
    X^\beta$,
    $$x\in\pi_{\alpha,\beta}[D]\quad\Leftrightarrow\quad
    \forall\gamma<\beta\ \ \
    x\har\gamma\in\pi_{\alpha,\gamma}[D].$$
  \end{itemize} 
  We say that $D$ is an \textit{$\bar\I$-full cube} if
  additionally we have
  \begin{itemize}
  \item[(i')] for each $\beta<\alpha$ and for each
    $x\in\pi_{\alpha,\beta}[D]$ the set
    $$(\pi_{\alpha,\beta+1}[D])_x\mbox{ is }\I_{\beta+1}\mbox{-full}.$$
  \end{itemize}
  If $\class$ is a projective pointclass, then we say that
  $D$ is an \textit{$\bar\I$-positive} (resp. \textit{full}) $\class$
    \textit{cube} if $D$ is $\bar\I$-positive (resp. full) cube and
  additionally
  \begin{itemize}
  \item for each $\beta\leq\alpha$ the set
    $\pi_{\alpha,\beta}[D]\in\class(\prod_{\gamma<\beta}
    X_\gamma)$.
  \end{itemize}
\end{definition}

Analogous definitions appear also in \cite{Kanovei},
\cite{CPA} and \cite{Zpl:FI}. Now we define Fubini products.

\begin{definition}\label{Fubini}
  Let $\alpha$ be a countable ordinal, $\langle
  X_\beta:\beta<\alpha\rangle$ be a sequence of Polish
  spaces and $\bar\I=\langle \I_\beta:\beta<\alpha\rangle$
  be a sequence of $\sigma$-ideals, $\I_\beta$ on $X_\beta$,
  respectively. A set $B\subseteq \prod_{\beta<\alpha}
  X_\beta$ belongs to $\bigotimes_{\beta<\alpha}\I_\beta$ if
  and only if there is an $\bar\I$-full cube $D\subseteq
  \prod_{\beta<\alpha} X_\beta$ disjoint from $B$.
\end{definition}

\begin{remark}\label{i2}  
  Let $X$ be a Polish space and let $\I$ be a $\coana$ on
  $\ana$ $\sigma$-ideal on $X$, generated by closed sets.
  Suppose $A\subseteq X^2$ is $\ana$. We will show that
  either $A$ belongs to $\I^2$, or else $A$ contains an
  $\I^2$-positive $\gdelta$ set.

  Let $D\subseteq X^2\times\baire$ be a closed set such that
  $\pi[D]=A$ (here $\pi$ denotes the projection to the first
  two coordinates).  By Lemma \ref{iperfectgdelta}, the
  family $\G$ of $\I$-perfect $\gdelta$ sets is $\ana$ (in
  the sense of Section \ref{sec:coanaidclosed}, in terms of
  $\tilde G$). Put $A'=\{x\in X: A_x\not\in \I\}$.  If
  $A'\in \I$, then clearly $A\in \I^2$.  Suppose that
  $A'\not\in \I$.  

  By Theorem~\ref{Sol.dich}, for each $x\in A'$ there is an
  $\I$-perfect $\gdelta$ set $G$ contained in $A_x$. Pick
  $x\in A'$ and such a $G\subseteq A_x$.  Using
  $\shoenfield$-absoluteness, we get a $\P_\I$-name $\dot y$
  for an element of $D$ such that (we identify $(X^2)_x$
  with $X$ here) $$G\Vdash\pi(\dot y)=\dot g$$ ($\dot g$ is
  the name for the generic point).

  Now, by properness and continuous reading of names for
  $\P_\I$, there is an $\I$-perfect $\gdelta$ set
  $G'\subseteq G$, and a continuous function
  $f:G'\rightarrow\baire$ with $f\subseteq D$. To see this,
  find a continuous function $f'$ reading $\dot y$, take
  suitable countable elementary submodel $M\prec H_\kappa$
  ($\kappa$ big enough), find $G'$ consisting of generic
  reals over $M$ and put $f=f'\har G$. The fact that
  $(x,f(x))\in D$ for $x\in G$ follows from
  $\ana$-absoluteness between $M[x]$ and $V$.

  Therefore, for each $x\in A'$ the following holds
  $$\exists G\in\G\quad \exists
  f:G\xrightarrow{*} D_x.$$ This is a $\ana$ formula, so by
  $\shoenfield$-absoluteness we have
  $$A'\Vdash\exists G\in\G\quad \exists
  f:G\xrightarrow{*} D_{\dot g}.$$ Again, by properness and
  continuous reading of names (applied to the name for (a
  code of) $\dom(f)$) we get an $\I$-perfect $\gdelta$ set
  $G'\subseteq A'$ and a continuous function
  $g:G'\rightarrow\G$ such that for each $x\in G'$ we have
  $\tilde G_{g(x)}\subseteq A_x$. Let $G=\{(x,y)\in X^2:
  x\in G'\wedge y\in \tilde G_{g(x)}\}=(g,id)^{-1}[\tilde
  G]$.  This is an $\I^2$-positive $\gdelta$ set contained
  in $A$.
\end{remark}

Note that the following lemma immediately follows from Lemma
\ref{i2coanaonana}

\begin{lemma}[Folklore]\label{incoanaonana}
  Suppose $\bar\I=\langle\I_k:k<n\rangle$ is a sequence of
  $\coana$ on $\ana$ $\sigma$-ideals, $\I_k$ on $X_k$. Let
  $A\subseteq \prod_{k<n} X_k$ be a $\ana$ set in
  $\bigotimes_{k<n}\I_k$. There is an $\bar\I$-full $\ana$
  cube $D$ disjoint from $A$.
\end{lemma}

\medskip

If $\alpha$ is a countable ordinal and
$\langle\I_\beta:\beta<\alpha\rangle$ is a sequence of
iterable $\sigma$-ideals, $\I_\beta$ on $X_\beta$, then we
denote by $*_{\beta<\alpha}\P_{\I_\beta}$ the countable
support iteration of $\P_{\I_\beta}$'s of length $\alpha$.
If $A\subseteq \prod_{\beta<\alpha} X_\beta$ is an
$\bar\I$-positive $\Bor$ cube, then we associate with $A$
the following condition $p_\alpha(A)$ in
$*_{\beta<\alpha}\P_{\I_\beta}$.  If $\beta<\alpha$, then
$p_\alpha(A)(\beta)$ is a
$*_{\gamma<\beta}\P_{\I_\gamma}$-name $\dot Y_\beta$ (for an
$\I_\beta$-positive Borel set) such that
$$p_\beta(\pi_{\alpha,\beta}[A])\Vdash \dot Y_\beta=A_{\dot
  g_\beta},$$ where $\dot g_\beta$ is the name for the
$*_{\gamma<\beta}\P_{\I_\gamma}$-generic point in
$\prod_{\gamma<\beta} X_\gamma$. Zapletal proved the
following (the statement in \cite[Theorem 5.1.6]{Zpl:FI}
deals with just one $\sigma$-ideal $\I$ but the proof shows
the stronger statement).

\begin{theorem}[{Zapletal, \cite[Theorem
    5.1.6]{Zpl:FI}}]\label{csiteration}
  Let $\alpha$ be a countable ordinal. If
  $\bar\I=\langle\I_\beta:\beta<\alpha\rangle$ is a sequence
  of iterable $\sigma$-ideals on Polish spaces $X_\beta$,
  respectively, then the function $p_\alpha$ is a dense
  embedding from the poset of $\bar\I$-positive $\Bor$ cubes
  (ordered by inclusion) into
  $*_{\beta<\alpha}\P_{\I_\beta}$. Moreover, any
  $\bigotimes_{\beta<\alpha}\I_\beta$-positive Borel set in
  $\prod_{\beta<\alpha} X_\beta$ contains an
  $\bar\I$-positive $\Bor$ cube and the forcing
  $\P_{\bigotimes_{\beta<\alpha}\I_\beta}$ is equivalent to
  $*_{\beta<\alpha}\P_{\I_\beta}$.
\end{theorem}

Kanovei and Zapletal proved also the following (again, the
statement of \cite[Theorem 5.1.9]{Zpl:FI} deals with one
$\sigma$-ideal but the proof generalizes to the statement
below).

\begin{theorem}[{Kanovei, Zapletal, \cite[Theorem
    5.1.9]{Zpl:FI}}]\label{KanZpl}
  Let $\alpha$ be a countable ordinal and
  $\bar\I=\langle\I_\beta:\beta<\alpha\rangle$ be a sequence
  of iterable $\sigma$-ideals on Polish spaces $X_\beta$,
  respectively. If $A\subseteq \prod_{\beta<\alpha}X_\beta$
  is $\ana$, then either $A\in
  \bigotimes_{\beta<\alpha}\I_\beta$, or else $A$ contains
  an $\bar\I$-positive $\Bor$ cube.
\end{theorem}

In the proof of Theorem \ref{KanZpl}, Kanovei and Zapletal
generalized Lemma \ref{incoanaonana} in the following way.

\begin{theorem}[{Kanovei, Zapletal, \cite[proof of Theorem
    5.1.9]{Zpl:FI}}]\label{ialphacoanaonana}
  Let $\alpha$ be a countable ordinal and
  $\bar\I=\langle\I_\beta:\beta<\alpha\rangle$ be a sequence
  of $\coana$ on $\ana$ $\sigma$-ideals on Polish spaces
  $X_\beta$, respectively. If $A\subseteq
  \prod_{\beta<\alpha} X_\beta$, is $\ana$ and
  $A\in\bigotimes_{\beta<\alpha}\I_\beta$, then there is an
  $\bar\I$-full $\ana$ cube $D$ disjoint from $A$.
\end{theorem}

The following (unpublished) corollary was communicated to me
by Pawlikowski.

\begin{corollary}[Pawlikowski, \cite{Janusz}]\label{Janusz}
  If $X$ is a Polish space and $\alpha$ is a countable
  ordinal, then $\M(X)^\alpha\cap\ana(X^\alpha)\subseteq\M(X^\alpha)$.
\end{corollary}

Theorem \ref{iterationgdelta} was motivated by Theorems
\ref{Sol.dich}, \ref{KanZpl} and Corollary \ref{Janusz}.
Now we restate it, in a slightly stronger version.

\begin{theorem}\label{iterationgdelta1}
  Let $\langle X_n: n<\omega\rangle$ be a sequence of Polish
  spaces and $\bar \I=\langle \I_n:n<\omega\rangle$ be a
  sequence of $\coana$ on $\ana$ $\sigma$-ideals generated
  by closed sets, $\I_n$ on $X_n$, respectively. If
  $A\subseteq\prod_{n<\omega} X_n$ is $\ana$, then
  \begin{itemize}
  \item either $A\in\bigotimes_{n<\omega}\I_n$,
  \item or else $A$ contains an $\bar\I$-positive $\gdelta$
    cube $G$ such that
  $$(\ana(\prod_{n<\omega} X_n)\cap \bigotimes_{n<\omega}\I_n)\har G\subseteq\M(G).$$
\end{itemize}
\end{theorem}

\begin{proof}
  Suppose $A\subseteq \prod_{n<\omega}X_n$ is an analytic
  $\bigotimes_{n<\omega}\I_n$-positive set. By Theorem
  \ref{KanZpl} we may assume that $A$ is an
  $\bar\I$-positive $\Bor$ cube. For each $n<\omega$ write
  $A_n$ for $\pi_{\omega,n}[A]$ and let $E_n\subseteq
  \prod_{i<n}X_i\times\baire$ be a closed set projecting to
  $A_n$. Let $\G_n\subseteq\baire$ be the analytic set from
  Lemma \ref{iperfectgdelta} consisting of codes of all
  $\I_n$-perfect $\gdelta$ sets. In this proof we denote
  $\pi_{n,n-1}$ by $\pi_n$ and write $\bar\I_n$ for
  $\langle\I_i:i<n\rangle$.

  \medskip 

  We will use the following lemma.
  \begin{lemma}[Kuratowski, Ulam]\label{KurUlam}
    Let $X$ and $Y$ be Polish spaces and let $f:Y\rightarrow
    X$ be a continuous open surjection. Suppose $B\subseteq
    Y$ has the Baire property and $$\forall^\M x\in X\quad
    B\cap f^{-1}[\{x\}]\mbox{ is meager in }
    f^{-1}[\{x\}].$$ Then $B$ is meager in $Y$.
  \end{lemma}

  \begin{proof}
    The proof is almost the same as the proof of the
    ``product'' version of the Kuratowski-Ulam theorem
    \cite[Theorem 8.41]{Kechris}. The difference is that
    instead of \cite[Lemma 8.42]{Kechris}, we need to prove
    that if $U\subseteq Y$ is open dense, then
    \begin{displaymath}
      \forall^{\M}x\in X\quad U\cap f^{-1}[\{x\}]\mbox{ is
        open dense in }f^{-1}[\{x\}].\tag{$*$}
    \end{displaymath}
    To show this, we take the open basis $\langle U_n:
    n<\omega\rangle$ of $Y$ and we show that for each
    $n<\omega$ the set $$V_n=\{x\in X:\ \ f^{-1}[\{x\}]\cap
    U_n=\emptyset\ \vee\ f^{-1}[\{x\}]\cap U_n\cap
    U\not=\emptyset\}$$ contains an open dense set. Indeed,
    let $W_n=X\setminus\cl{f[U_n]}$ and notice that the set
    $f[U_n]\cup W_n$ is open dense in $X$. Moreover,
    $W_n\subseteq V_n$ and $V_n\cap f[U_n]$ is dense open in
    $f[U_n]$. Now, notice that if
    $x\in\bigcap_{n<\omega}V_n$, then $U\cap f^{-1}[\{x\}]$
    is open dense in $f^{-1}[\{x\}]$. This proves $(*)$.
  \end{proof}

  \medskip

  We shall construct a sequence of $\gdelta$ sets
  $G_n\subseteq\prod_{i<n}X_i$ such that
  \begin{itemize}
  \item[(i)] $\pi_n[G_n]\subseteq G_{n-1}$ is comeager in
    $G_{n-1}$,
  \item[(ii)] for each $x\in\pi_n[G_n]$ the set $(G_n)_x$ is
    $\I_n$-perfect,
  \item[(iii)] $\pi_n\har G_n:G_n\rightarrow \pi_n[G_n]$ is
    an open map,
  \item[(iv)] if $D\subseteq \prod_{i<n}X_i$ is an
    $\bar\I_n$-full $\ana$ cube, then $D\cap G_n$ is
    comeager in $G_n$
  \end{itemize}

  Note that, by Lemma \ref{incoanaonana}, (iv) implies
  \begin{itemize}
  \item[(v)] $(\ana(\prod_{i<n}X_i)\cap
    \bigotimes_{i<n}\I_i)\har G_n\subseteq\M(G_n)$.
  \end{itemize}

 \medskip

 For $n=0$ use Lemma \ref{iperfectdense} to find an
 $\I_0$-perfect $\gdelta$-set $G_0\subseteq A_0$. Notice
 that $(\ana(X_0)\cap\I_0)\har G_0\subseteq \M(G_0)$ follows
 from the fact that $G_0$ is $\I_0$-perfect.

  \medskip

  Suppose the set $G_n\subseteq X^n$ is constructed.
  Similarly as in Remark \ref{i2} we conclude that by Lemma
  \ref{iperfectdense}, $\shoenfield$-absoluteness and
  continuous reading of names for $\P_\I$, for each $x\in
  G_n$ there is a code $c(x)\in\G_{n+1}$ for an
  $\I_{n+1}$-perfect $\gdelta$ set $\tilde G_{c(x)}$ and a
  function $f:\tilde G_{c(x)}\xrightarrow{*} (E_{n+1})_x$.
  Consider the set
  \begin{displaymath}
    W=\{(x,d)\in X^n\times\baire:\ \ x\in G_n,\ 
    \exists c\in\G_{n+1}\, \exists f:\tilde
    G_c\xrightarrow{*} (E_{n+1})_x\ \ \dom(f)=\tilde
    G_d\}.
  \end{displaymath} 
  $W$ is analytic and all vertical sections of $W$ are
  nonempty. Hence, by the Jankov-von Neumann theorem, $W$
  has a $\sigma(\ana)$-measurable uniformization
  $g:G_n\rightarrow\G_{n+1}$. In particular, $g$ is Baire
  measurable and hence it is continuous on a dense $\gdelta$
  set $G_n'\subseteq G_n$. Let
  $$G_{n+1}=\{(x,y)\in X^n\times X: x\in
  G_n'\,\wedge\,y\in\tilde G_{g(x)}\}.$$ $G_{n+1}$ is a
  $\gdelta$ set since $G_{n+1}=(g,id)^{-1}[\tilde G]$.
  Moreover, $\pi_{n+1}[G_{n+1}]=G_n'$ is comeager in $G_n$.

  Note that the function $\pi_{n+1}\har G_{n+1}:
  G_{n+1}\rightarrow G_n'$ is open by Lemma
  \ref{gdeltaeffros} and the fact that $g$ is continuous on
  $G_n'$.

  Now, let $D\subseteq \prod_{i<n+1}X_i$ be an
  $\bar\I_{n+1}$-full $\ana$ cube. The set
  $D_n=\pi_{n+1}[D]$ is an $\bar\I_n$-full $\ana$ cube, so,
  by the inductive hypothesis, $D_n\cap G_n$ is comeager in
  $G_n$.  Therefore, $D_n\cap G_n'$ is comeager in $G_n'$.
  Moreover, if $x\in G_n'\cap D_n$, then
  $D_x\cap(G_{n+1})_x$ is comeager in $(G_{n+1})_x$, since
  $(G_{n+1})_x$ is $\I_{n+1}$-perfect. Now, $D$ has the
  Baire property, so by Lemma \ref{KurUlam} (for the
  function $\pi_{n+1}\har G_{n+1}: G_{n+1}\rightarrow G_n'$)
  we have that $D\cap G_{n+1}$ is comeager in $G_{n+1}$.

  This ends the construction.

  \medskip

  Put
  $$G=\bigcap_{n<\omega} \pi_{\omega,n}^{-1}[G_n].$$ $G$ is
  a $\gdelta$ set and it is contained in $A$ since $A$ is a
  ($\bar\I$-positive) cube. For each $n,k<\omega$ consider
  also the set
  \begin{eqnarray*}
    H_n^k=\{x_n\in G_n:\ \forall^\M y_{n+1}\in
    (G_{n+1})_{x_n}\ \ldots\ \forall^\M y_{n+k}\in
    (G_{n+k})_{(x_n,y_{n+1},\ldots,y_{n+k-1})}\\
    \exists y_{n+k+1}\in
    (G_{n+k+1})_{(x_n,y_{n+1},\ldots,y_{n+k})}
    \quad(x_n,y_{n+1},\ldots y_{n+k+1})\in G_{n+k}\}
  \end{eqnarray*}

  Applying $(k+1)$-many times Lemma \ref{KurUlam} we
  conclude that $H_n^k$ is comeager in $G_n$ for each
  $k<\omega$. Put $$H_n=\bigcap_{k<\omega} H_n^k.$$ Each
  $H_n$ is also a comeager subset of $G_n$. Notice that
  $\pi_{n+1}[H_{n+1}]\subseteq H_n$ and $\pi_{n+1}[H_{n+1}]$
  is comeager in $G_n$ for each $n<\omega$. Moreover, for
  each $n<m<\omega$
  \begin{displaymath}
    \pi_{m,n}[H_m]\mbox{ is comeager in }G_n\tag{$*$}    
  \end{displaymath}
  (by repeatedly applying Lemma \ref{KurUlam}). 

  Notice that for each $n<\omega$ we have
  $$\bigcap_{n<m<\omega}\pi_{m,n}[H_m]\subseteq
  \pi_{\omega,n}[G].$$ Consequently, by $(*)$ we have that
  $\pi_{\omega,n}[G]$ is comeager in $G_n$. Therefore, it is
  $\bigotimes_{i<n}\I_i$-positive, by (v). For each
  $k<\omega$ and $x\in\pi_{\omega,k}[G]$ we may repeat the
  above argument in the space $(\prod_{n<\omega}X_n)_x$ and
  conclude that the set $(\pi_{\omega,k+1}[G])_x$ is
  comeager in $(G_{k+1})_x$ and hence is $\I_{k+1}$-positive
  (since $(G_{k+1})_x$ is $\I_{k+1}$-perfect). Therefore $G$
  is $\bar\I$-positive $\gdelta$ cube.

  \medskip

  Now we prove that $(\ana(\prod_{n<\omega} X_n)\cap
  \bigotimes_{n<\omega}\I_n)\har G\subseteq\M(G)$. By
  Theorem \ref{ialphacoanaonana} it is enough to prove that
  if $D\subseteq\prod_{n<\omega}X_n$ is $\bar\I$-full $\ana$
  cube, then $D\cap G$ is comeager in $G$. Write
  $D_n=\pi_{\omega,n}[D]$.  Using (iv) we see that $D_n$ is
  comeager in $G_n$. For each $n<\omega$ find a dense in
  $G_n$ $\gdelta$ set $G_n''\subseteq D_n$ such that
  $G_{n+1}''\subseteq\pi_n^{-1}[G_n'']$. Let
  $G''=\bigcap_{n<\omega}\pi_{\omega,n}^{-1}[G_n'']$. Note
  that $G''\subseteq D\cap G$ and $G''$ is a $\gdelta$ set.
  We will prove that $G''$ is dense in $G$.

  Repeatedly applying Theorem \ref{KurUlam} and the property
  (iv) we see that for each $n<\omega$ the following holds
  $$\forall^\M y_0\in G_0\ \forall^\M y_1\in (G_1)_{y_0}\
  \ldots \forall^\M y_n\in(G_n)_{(y_0,\ldots,y_{n-1})}\quad
  (y_0,\ldots,y_n)\in G_n''.$$ Using this we can easily show
  $G''$ is nonempty, and, in fact, that if $U_n\subseteq
  X^n$ is open, then $G''\cap\pi_{\omega,n}^{-1}[U_n]$ is
  nonempty. But this implies that $G''$ is dense in $G$.

  This ends the proof.
\end{proof}

Let $X$ and $Y$ be Polish spaces and
$F:X\rightarrow\mathcal{P}(Y)$ be a multifunction. If
$\mathcal{A}$ is a family of subsets of $X$, then we say
that $F$ is \textit{$\mathcal{A}$-measurable} if for each
open set $U\subseteq Y$ the set $F^{-1}(U)=\{x\in X:
F(x)\cap U\not=\emptyset\}$ belongs to $\mathcal{A}$. We say
that $F$ is an \textit{analytic multifunction} if its graph,
i.e $\bigcup_{x\in X}\,\{x\}\times F(x)$, is analytic in
$X\times Y$. The following result is motivated by the
Kuratowski-Ryll Nardzewski theorem.

\begin{proposition}
  Let $X$ be a Polish space and $\I$ a $\sigma$-ideal on $X$
  generated by closed sets. If
  $F:X\rightarrow\mathcal{P}(\baire)$ is an analytic
  multifunction then there is an $\I$-positive $\gdelta$ set
  $G$ such that $F\har G$ is
  $\mathbf{\Sigma}^0_3$-measurable.
\end{proposition}
\begin{proof}
  Denote the graph of $F$ by $A$ and let $a$ be such that
  $A\in\lana(a)$. Let $A(v,w)$ be a $\lana(a)$ formula
  defining the set $A$. Take $M\prec H_\kappa$ (for a big
  enough $\kappa$) containing $a$ and $\P_\I$. Let
  $Gen(M)\subseteq X$ be the set of all $\P_\I$-generic
  reals over $M$. $Gen(M)$ is an $\I$-positive Borel set by
  properness of $\P_\I$. Find an $\I$-positive $\gdelta$ set
  $G\subseteq Gen(M)$. We will show that $F\har G$ is
  $\mathbf{\Sigma}^0_3$-measurable.  Notice that if
  $\tau\in\otr$ and $x\in X$, then
  $$x\in F^{-1}([\tau])\quad\mbox{iff}\quad\exists y\in[\tau]\
  A(x,y).$$ This is a $\lana(a)$ formula, so it is absolute
  for $M[x]\subseteq V$. Therefore, by a usual forcing
  argument and the fact that $\I$-positive $\gdelta$ sets
  are dense in $\P_\I$ we get
  $$F^{-1}\big[[\tau]\big]=\bigcup\{G\in\P_\I\cap
  M:\quad G\in\gdelta\ \wedge\ G\Vdash\exists y\in[\tau]\
  A(\dot g,y)\}.$$ This is a $\mathbf{\Sigma}^0_3$ set.
\end{proof}

\section{Closed null sets}
\label{sec:E}

We denote by $\E$ the $\sigma$-ideal generated by closed
null sets in $\cantor$ (with respect to the standard Haar
measure $\mu$ on $\cantor$). The sets in $\E$ are both null
and meager. $\E$ is properly contained in $\M\cap\N$
\cite[Lemma 2.6.1]{BaJu} and in fact, one can show that $\E$
is not ccc.

The family of closed sets in $\E$ coincides with the family
of closed null sets and $\E\cap\closed(\cantor)$ is a
$\gdelta$ set in $K(\cantor)$ (for each $\varepsilon>0$ the
set $\{C\in K(\cantor): \mu(C)<\varepsilon\}$ is open).
Therefore, $\E$ is $\coana$ on $\ana$ by Corollary
\ref{coanaonana}.

The forcing $\P_\E$ adds an unbounded real and a splitting
real. In fact, one can check that the generic real is
splitting. To see that $\P_\E$ adds an unbounded real,
recall a theorem of Zapletal \cite[Theorem 3.3.2]{Zpl:FI},
which says that a forcing $\P_\I$ is $\baire$-bounding if and
only $\P_\I$ has continuous reading of names and compact sets
are dense in $\P_\I$. Let $G$ be a $\gdelta$ set such that
$G\in \N$ and $\cantor\setminus G\in \M$. $G$ is
$\E$-positive but no compact $\E$-positive set is contained
in $G$. In particular compact sets are not dense in $\P_\E$
and hence this forcing is not $\baire$-bounding. $\P_\E$
does not, however, add a dominating real. This follows from
another theorem of Zapletal \cite[Theorem 3.8.15]{Zpl:FI},
which says that if $\I$ is a $\coana$ on $\ana$
$\sigma$-ideal, then the forcing $\P_\I$ does not add a
dominating real.

Zapletal proved in \cite[Theorem 4.1.7]{Zpl:FI} (see the
first paragraph of the proof) that if $\I$ is a
$\sigma$-ideal generated by an analytic collection of closed
sets and $V\subseteq V[G]$ is a $\P_\I$-extension, then any
intermediate extension $V\subseteq W\subseteq V[G]$ is equal
either to $V$ or $V[G]$, or is an extension by a Cohen real.
Therefore it is natural to ask if $\P_\E$ adds Cohen reals.

Recall that a closed set $D\subseteq\cantor$ is
\textit{self-supporting} if for any clopen set $U$ the set
$D\cap U$ is either empty or not null. Notice that a closed
set is self-supporting if and only if it is $\E$-perfect.
If $\mu$ is a Borel measure on $X$ and $A\in\Bor(X)$ is such
that $\mu(A)>0$, then by $\mu_A$ we denote the
\textit{relative measure} on $A$ defined as
$\mu_A(B)=\mu(A\cap B)\slash\mu(A)$.

\begin{theorem}\label{nocohene}
  The forcing $\P_\E$ does not add Cohen reals.
\end{theorem}

\begin{proof}
  Suppose $B\in\P_\E$ and $\dot x$ is a name for a real such
  that $$B\Vdash\dot x\mbox{ is a Cohen real}.$$ By Lemma
  \ref{iperfectdense} and continuous reading of names we
  find a $\gdelta$ set $G\subseteq B$ such that $D=\cl{G}$
  is self-supporting and a continuous function
  $f:G\rightarrow\baire$ such that $G\Vdash \dot x= f(\dot
  g)$. Pick a continuous, strictly positive measure $\nu$ on
  $G$. For each $\tau\in\otr$ the set
  $C_\tau=f^{-1}\big[[\tau]\big]$ is a relative clopen in
  $G$. Find open sets $C_\tau'\subseteq D$ such that
  $C_\tau=C_\tau'\cap G$. $D$ is zero-dimensional, so by the
  reduction property for open sets we may assume that
  \begin{itemize}
  \item $C_{\tau_0}'\subseteq C_{\tau_1}'$ for
    $\tau_0\subseteq\tau_1$,
  \item $C_{\tau_0}'\cap C_{\tau_1}'=\emptyset$ for
    $\tau_0\perp\tau_1$.
  \end{itemize}
  We will find a tree $T\subseteq\otr$ such that $\lim T$ is
  nowhere dense in $\baire$ and the closure of the set
  $f^{-1}[\lim T]$ is self-supporting.

  Enumerate all nonempty clopen sets in $D$ in a sequence
  $\langle V_n':n<\omega\rangle$ and all nonempty clopen
  sets in $G$ in a sequence $\langle V_n:n<\omega\rangle$,
  and elements of $\otr$ in a sequence
  $\langle\sigma_n:n<\omega\rangle$. If $\tau\in\otr$, then
  $\langle C_{\tau^\smallfrown n}':n<\omega\rangle$ is a
  sequence of disjoint open sets in $D$ and $\langle
  C_{\tau^\smallfrown n}:n<\omega\rangle$ is a sequence of
  disjoint open sets in $G$. Thus for each $\varepsilon>0$
  there is $n\in\omega$ such that $\mu(C_{\tau^\smallfrown
    n}')<\varepsilon$ as well as $\nu(C_{\tau^\smallfrown
    n})<\varepsilon$.  Moreover, for each $m\in\omega$ there
  is $n\in\omega$ such that $\mu_{V_m'}(V_m'\cap
  C_{\tau^\smallfrown n}')<\varepsilon$ and
  $\nu_{V_m}(V_m\cap C_{\tau^\smallfrown n})<\varepsilon$.

  By induction, we find a collection of nodes
  $\tau_n\in\otr$ such that the tree
  $$T=\{\tau\in\otr: \forall n\ \tau_n\not\subseteq\tau\}$$
  is such that $\lim T$ is nowhere dense, and for each
  $m<\omega$ we have
  $$\mbox{either\quad} V_m'\subseteq \bigcup_{n<\omega}
  C_{\tau_n}'\mbox{\quad or\quad }\mu(V_m\setminus
  \bigcup_{n<\omega} C_{\tau_n}')>0$$ and
  $$\mbox{either\quad} V_m\subseteq \bigcup_{n<\omega}
  C_{\tau_n}\mbox{\quad or\quad} \nu(V_m\setminus
  \bigcup_{n<\omega} C_{\tau_n})>0.$$ Along the induction we
  also construct sequences of reals $\varepsilon_n\geq0$ and
  $\delta_n\geq0$. 

  At the $n$-th step of the induction consider the sets
  $U_n'=V_n'\setminus\bigcup_{i<n} C_{\tau_i}'$ and
  $U_n=V_n\setminus\bigcup_{i<n} C_{\tau_i}$, which are
  either empty or of positive measure ($\mu$ or $\nu$,
  respectively) by the inductive assumption. Put
  $\varepsilon_n=\mu(U_n')$, $\delta_n=\mu(U_n)$. Find
  $\tau_n\in\otr$ such that $\tau_n={\sigma_n}^\smallfrown
  k$ for some $k<\omega$ and for each $i\leq n$
  \begin{itemize}
  \item if $\varepsilon_i>0$, then
    $\mu_{V_i'}(C_{\tau_n}'\cap
    V_i')<2^{-n-1}\varepsilon_i$,
  \item if $\delta_i>0$, then $\nu_{V_i}(C_{\tau_n}\cap
    V_i)<2^{-n-1}\delta_i$.
  \end{itemize}

  The set
  $$A=G\setminus\bigcup_n C_{\tau_n}$$ is of type $\gdelta$.
  Moreover, it follows from the construction that
  $\cl{A}=D\setminus\bigcup_n C_{\tau_n}'$ and that $\cl{A}$
  is self-supported, so $A\not\in\E$ by Lemma
  \ref{iperfect}.  On the other hand, $A\Vdash\dot x\in\lim
  T$, which gives a contradiction, since $\lim T$ is nowhere
  dense.
\end{proof}

\begin{corollary}\label{eminimal}
  If $G$ is $\P_\E$-generic over $V$, then the extension
  $V\subseteq V[G]$ is minimal.
\end{corollary}

Now we will introduce a fusion scheme for the $\sigma$-ideal
$\E$.  Denote by $G_\E$ the following game scheme. In his
$n$-th turn, Adam picks $\xi_n\in 2^n$ such that
$\xi_n\supsetneq\xi_{n-1}$ ($\xi_{-1}=\emptyset$). In her
$n$-th turn, Eve picks a basic clopen set $C_n\subseteq
[\xi_n]$ such that
$$\mu_{[\xi_n]}(C_n)<\frac{1}{n}.$$

For a set $A\subseteq\cantor$ we define the game $G_\E(A)$
in the game scheme $G_\E$ as follows. Eve wins a play in
$G_\E(A)$ if
$$x\not\in A\quad\vee\quad\forall^\infty
n\ x\in C_n$$ (where $x\in\cantor$ is the union of the
$\xi_n$'s picked by Adam). Otherwise Adam wins.

\begin{proposition}\label{gamechare}
  For any set $A\subseteq \cantor$, Eve has a winning
  strategy in $G_\E(A)$ if and only if $A\in\E$.
\end{proposition}
\begin{proof}
  Suppose first that Eve has a winning strategy $S$ in
  $G_\E(A)$. For each $\sigma\in 2^{<\omega}$ consider a
  partial play $\tau_\sigma$ in which Adam picks
  successively $\sigma\har k$ for $k\leq |\sigma|$. Let
  $C_\sigma$ be the Eve's next move, according to $S$, after
  $\tau_\sigma$.  Put $E_n=\bigcup_{\sigma\in 2^n}
  C_\sigma$. Clearly $E_n$ is a clopen set and $\mu(E_n)\leq
  1\slash n$. Let $D_n=\bigcap_{m\geq n} E_n$. Now, each
  $D_n$ is a closed null set and $A\subseteq \bigcup_n D_n$
  since $S$ is a winning strategy. Therefore $A\in\E$.

  Conversely, assume that $A\in\E$. There are closed null
  sets $D_n$ such that $A\subseteq\bigcup_n D_n$. Without
  loss of generality assume $D_n\subseteq D_{n+1}$. Let
  $T_n\subseteq\omega^{<\omega}$ be a tree such that
  $D_n=\lim T_n$. We define a strategy $S$ for Eve as
  follows. Suppose Adam has picked $\sigma\in 2^n$ in his
  $n$-th move and consider the tree $T_n(\sigma)$. Since
  $\lim T_n(\sigma)$ is of measure zero, there is $k<\omega$
  such that
  $$\frac{|T_n(\sigma)\cap 2^k|}{2^k}<\frac{1}{n}.$$ Let Eve's
  answer be the set $\bigcup_{\tau\in T_n(\sigma)\cap
    2^k}\,[\tau]$.  One can readily check that this defines
  a winning strategy for Eve in $G_\E(A)$.
\end{proof}

\begin{corollary}\label{coregame}
  If $B\subseteq\cantor$ is Borel, then $B\in\E$ if and only
  if Eve has a winning strategy in $G_\E(B)$.
\end{corollary}

\section{Decomposing Baire class 1
  functions}\label{sec:Piececon}

Let $X$ and $Y$ be Polish spaces and $f:X\rightarrow Y$ be a
Borel function. We say that $f$ is \textit{piecewise
  continuous} if $X$ can be covered by a countable family of
closed sets on each of which $f$ is continuous.

Recall that a function is $\gdelta$-measurable if preimages
of $\gdelta$ sets are $\gdelta$ or, equivalently, preimages
of open sets are $\gdelta$. If $f:X\rightarrow Y$ is a
$\gdelta$-measurable function, then preimages of closed sets
are $\fsigma$. Therefore, if $Y$ is zero-dimensional, then
preimages of open sets are also $\fsigma$, so consequently
$\mathbf{\Delta}^0_2$.

The following characterization of piecewise continuity has
been given by Jayne and Rogers.

\begin{theorem}[{Jayne, Rogers, \cite[Theorem
    5]{JR}}]\label{jrthm}
  Let $X$ be a Souslin space and $Y$ be a Polish space. A
  function $f:X\rightarrow Y$ is piecewise continuous if and
  only if it is $\gdelta$-measurable.
\end{theorem}

A nice and short proof of the Jayne-Rogers theorem can be
found in \cite{BSLMR}. Classical examples of Borel functions
which are not piecewise continuous are the Lebesgue
functions $L,L_1:\cantor\rightarrow\R$ (for definitions see
\cite[Section 1]{Sol.Dec}). For two functions
$f:X\rightarrow Y$ and $f':X'\rightarrow Y'$ we write
$f\sqsubseteq f'$ if there are topological embeddings
$\varphi:X\rightarrow X'$ and $\psi:Y\rightarrow Y'$ such
that $f'\circ\varphi=\psi\circ f$. In \cite{Sol.Dec} Solecki
strengthened Theorem \ref{jrthm} proving the following
result.

\begin{theorem}[{Solecki, \cite[Theorem 3.1]{Sol.Dec}}]\label{soldec}
  Let $X$ be a Souslin space, $Y$ be a Polish space and
  $f:X\rightarrow Y$ be Baire class 1. Then
  \begin{itemize}
  \item either $f$ is piecewise continuous,
  \item or $L\sqsubseteq f$, or $L_1\sqsubseteq f$.
  \end{itemize}
\end{theorem}

Theorem \ref{jrthm} follows from Theorem \ref{soldec}
because neither $L$ nor $L_1$ is $\gdelta$-measurable.

From now now until the end of this section we fix a Polish
space $X$ and a Baire class 1, not piecewise continuous
function $f:X\rightarrow\baire$ ($\baire$ can be replaced
with any zero-dimensional Polish space). Consider the
$\sigma$-ideal $\I^f$ on $X$ generated by closed sets on
which $f$ is continuous. We will prove that the forcing
$\P_{\I^f}$ is equivalent to the Miller forcing (see
Corollary \ref{piececonmiller}).

Suppose $C\subseteq X$ is a compact set and
$c:\cantor\rightarrow C$ is a homeomorphism. We call $(c,C)$
\textit{a copy of the Cantor space} and denote it by
$c:\cantor\hookrightarrow C\subseteq X$.  We denote by $\Q$
the set of all points in $\cantor$ which are eventually
equal to $0$.

\begin{proposition}\label{piececon}
  Suppose that $G\subseteq X$ is a $\gdelta$ set such that
  $G\not\in \I^f$. There exist an open set
  $U\subseteq\baire$ and a copy of the Cantor space
  $c:\cantor\hookrightarrow C\subseteq X$ such that
  \begin{itemize}
  \item $f^{-1}[U]\cap C=c[\Q]$,
  \item $C\setminus c[\Q]\subseteq G$.
  \end{itemize}
\end{proposition}

\begin{proof}
  Denote $\{\tau\in\ctr:
  \tau=\emptyset\vee\tau(|\tau|-1)=0\}$ by $Q$.

  \begin{definition}
    A \textit{Hurewicz scheme} is a Cantor scheme of closed
    sets $F_\tau\subseteq X$ for $\tau\in\ctr$ together with
    a family of points $x_\tau\in X$ and clopen sets
    $U_\tau\subseteq\baire$ for $\tau\in Q$ such that:
    \begin{itemize}
    \item $x_{\tau^\smallfrown 0}=x_\tau$,
      $U_{\tau^\smallfrown 0}=U_\tau$,
    \item $x_\tau\in F_\tau \cap f^{-1}[U_\tau]$.
    \end{itemize}
  \end{definition}
  Suppose that $G=\bigcap_n G_n$ with each $G_n$ open and
  $G_{n+1}\subseteq G_n$. We will construct a Hurewicz
  scheme such that for each $\tau\in 2^n$ the following two
  conditions hold:
  \begin{itemize}
  \item $\Big(F_\tau\setminus f^{-1}\big[\bigcup_{\sigma\in
      2^n\cap Q}U_\sigma\big]\Big)\cap G\not\in \I^f$,
  \item $F_{\tau^\smallfrown 1}\subseteq
    \Big(F_\tau\setminus f^{-1}\big[\bigcup_{\sigma\in
      2^n\cap Q} U_\sigma\big]\Big)\cap G_{|\tau|}$.
  \end{itemize}
  We need the following lemma (its special case can be found
  in the proof of the Jayne-Rogers theorem in \cite{BSLMR}).
  \begin{lemma}\label{jr}
    If $F$ is a closed set in $X$ and $F\cap G\not\in \I^f$
    then there is $x\in F$ and a clopen set
    $U\subseteq\baire$ such that $f(x)\in U$ and for each
    open neighborhood $V$ of $x$ $$(V\setminus
    f^{-1}[U])\cap G\,\not\in \I^f.$$ Moreover, if
    $W\subseteq\baire$ is a clopen set such that $F\cap
    f^{-1}[W]\not\in \I^f$, then we may require that
    $U\subseteq W$.
  \end{lemma}
  \begin{proof}
    First let $W=\baire$. Without loss of generality assume
    that for each nonempty open set $V\subseteq F$ we have
    $V\cap G\not\in \I^f$.  Supppose that the conclusion is
    false. We show that $f$ is continuous on $F$,
    contradicting the fact that $F\cap G\not\in \I^f$. Pick
    arbitrary $x\in F$ and a clopen set $U$ such that
    $f(x)\in U$. By the assumption there is an open
    neighborhood $V\ni x$ such that $(V\setminus
    f^{-1}[U])\cap G\,\in \I^f$. We claim that $V\subseteq
    f^{-1}[U]$. Suppose otherwise, then there is $y\in V$
    such that $f(y)\not\in U$. Pick a clopen set $U'$ such
    that $U'\cap U=\emptyset$ and $f(y)\in U'$.  Again, by
    the assumption there is an open neighborhood $V'$ of $y$
    such that $(V'\setminus f^{-1}[U'])\cap G\,\in \I^f$.
    Now $V''=V\cap V'$ is a nonempty open set and since
    $U'\cap U=\emptyset$ we have that $$V''\cap G\
    \subseteq\ (V\setminus f^{-1}[U])\cap G\ \cup\
    (V'\setminus f^{-1}[U'])\cap G.$$ This shows that
    $V''\cap G\,\in \I^f$, a contradiction.

    Now, if $W\subseteq\baire$ is a clopen set such that
    $F\cap f^{-1}[W]\not\in \I^f$ then $f^{-1}[W]\cap F$ is
    an $\fsigma$ set since $f$ is Baire class 1. So there is
    a closed set $F'\subseteq F$ such that $F'\not\in \I^f$
    and $f[F']\subseteq W$. Applying the previous argument
    to $F'$ we get a clopen set $U$ such that $U\subseteq
    W$.  This ends the proof.
  \end{proof}

  Now we construct a Hurewicz scheme. First use Lemma
  \ref{jr} to find $x_\emptyset$, $U_\emptyset$ and put
  $F_\emptyset=X$.  Suppose the scheme is constructed up to
  the level $n-1$. 

  First we will construct $U_{\sigma^\smallfrown 0}$ and
  $x_{\sigma^\smallfrown 0}$ for each
  $\sigma\in2^{n-1}\setminus Q$ (recall that for
  $\sigma\in2^{n-1}\cap Q$ we put $U_{\sigma^\smallfrown
    0}=U_\sigma$ and $x_{\sigma^\smallfrown 0}=x_\sigma$).

  For each $\sigma\in2^{n-1}\setminus Q$ find a nonempty,
  perfect closed set $C_\sigma\subseteq\baire$ such that
  $F_\sigma\cap f^{-1}[V]\,\not\in \I^f$ for each nonempty
  relatively clopen set $V\subseteq C_\sigma$ (this is done
  by removing from $\baire$ those clopen sets $U$ such that
  $F_\tau\cap f^{-1}[U]\in \I^f$).

  \begin{lemma}
    There is a sequence of nonempty clopen (in $\baire$)
    sets $\langle W_\tau: \tau\in2^{n-1}\setminus Q\rangle$
    such that
    \begin{itemize}
    \item $W_\tau\cap C_\tau\not=\emptyset$
    \item for each $\sigma\in 2^{n-1}\cap Q$ and for each
      open neighborhood $V$ of $x_\sigma$ we have
      $$\bigg(\big(V\setminus\bigcup_{\sigma'\in 2^{n-1}\cap Q}
      f^{-1}[U_{\sigma'}]\big)\setminus\bigcup_{\tau\in
        2^{n-1}\setminus Q} f^{-1}[W_\tau]\bigg)\cap G\,\not\in
      \I^f.$$
    \end{itemize}    
  \end{lemma}
  \begin{proof}
    Enumerate $2^{n-1}\setminus Q$ in a sequence
    $\langle\tau_i:i<2^{n-2}\rangle$ and construct the sets
    $W_{\tau_i}$ by induction on $i<2^{n-2}$. Fix
    $i<2^{n-2}$ and suppose that $W_{\tau_j}$ are already
    defined for $j<i$ and
    \begin{displaymath}
      \bigg(\big(V\setminus\bigcup_{\sigma\in 2^{n-1}\cap Q}
      f^{-1}[U_\sigma]\big)\setminus\bigcup_{j<i}
      f^{-1}[W_{\tau_j}]\bigg)\cap G\,\not\in \I^f.
    \end{displaymath}

    \begin{claim}
      If $O_0$ and $O_1$ are two disjoint nonempty clopen
      sets in $\baire$, then for each $\sigma\in2^{n-1}\cap
      Q$ there exists $k\in\{0,1\}$ such that for each open
      neighborhood $V$ of $x_\sigma$ the following holds
      \begin{displaymath}
        \bigg(\big(V\setminus\bigcup_{\sigma\in 2^{n-1}\cap Q}
        f^{-1}[U_\sigma]\big)\setminus\bigcup_{j<i}
        f^{-1}[W_{\tau_j}]\bigg)\cap G\setminus
        f^{-1}[O_k]\,\not\in \I^f.
      \end{displaymath}      
    \end{claim}
    \begin{proof}
      Notice that for a single open neighborhood $V$ of
      $x_\sigma$ one $k\in\{0,1\}$ is good. If $\langle
      V_n:n<\omega\rangle$ is a base at $x_\sigma$, then
      some $k\in\{0,1\}$ is good for infinitely many of
      them.
    \end{proof}
    Enumerate $2^{n-1}\cap Q$ in a sequence
    $\langle\sigma_k:k<2^{n-2}\rangle$. Using the above
    Claim and the fact that $C_{\tau_i}$ is perfect, find a
    decreasing sequence of nonempty clopen sets
    $O_k\subseteq \baire$ for $k<2^{n-2}$ such that
    \begin{itemize}
    \item $O_k\subseteq O_{k-1}$,
    \item $O_k\cap C_{\tau_i}\not=\emptyset$,
    \item for each open neighborhood $V$ of $x_{\sigma_k}$
      \begin{displaymath}
        \bigg(\big(V\setminus\bigcup_{\sigma\in 2^{n-1}\cap Q}
        f^{-1}[U_\sigma]\big)\setminus\bigcup_{j<i}
        f^{-1}[W_{\tau_j}]\bigg)\cap G\setminus
        f^{-1}[O_k]\,\not\in \I^f.
      \end{displaymath}
    \end{itemize}
    Finally, let $W_{\tau_i}$ be the last of $O_k$'s.
  \end{proof}

  By the assumption on $C_\tau$'s, we have $F_\tau\cap
  f^{-1}[W_\tau]\not\in \I^f$, for each
  $\tau\in2^{n-1}\setminus Q$. Using Lemma \ref{jr}, for
  each $\tau\in 2^{n-1}\setminus Q$ find a clopen set
  $U_{\tau^\smallfrown 0}\subseteq W_\tau$ and a point
  $x_{\tau^\smallfrown 0}$ such that the assertion of Lemma
  \ref{jr} holds.

  Now all $U_\tau$ and $x_\tau$ for $\tau\in2^n\cap Q$ are
  defined and we need to find sets $F_\sigma$ for
  $\sigma\in2^n$.
  \begin{claim}
    For each $\sigma\in2^{n-1}$ there are two disjoint
    $\I^f$-positive closed sets $F_{\sigma^\smallfrown 0},
    F_{\sigma^\smallfrown 1} \subseteq F_\sigma$ of
    diameters less than $1\slash n$ such that
    $$F_{\sigma^\smallfrown 1}\subseteq
    \big(F_\sigma\setminus\bigcup_{\tau\in 2^n\cap
      Q}f^{-1}[U_\tau]\big)\cap G_{n-1}$$ and
    $F_{\sigma^\smallfrown 0}$ contains
    $x_{\sigma^\smallfrown 0}$.
  \end{claim}
  \begin{proof}
    For each $\sigma\in 2^{n-1}$ take an open neighborhood
    $V_\sigma$ of $x_{\sigma^\smallfrown 0}$ of diameter
    $<1\slash n$. The set
    $$V_\sigma\setminus f^{-1}\big[\bigcup_{\tau\in2^n\cap Q}
    U_\tau\big]$$ is $\fsigma$ (since $f$ is Baire class 1)
    which has $\I^f$-positive intersection with $G$. Thus it
    has a closed subset $F$ such that $F$ also has
    $\I^f$-positive intersection with $G$.  Now, the set
    $F\cap G_{n-1}$ is $\fsigma$, so find
    $F_{\sigma^\smallfrown 1}$ which is a closed subset of
    $F\cap G_{n-1}$ and has $\I^f$-positive intersection
    with $F\cap G$. Let $F_{\sigma^\smallfrown 0}$ be a
    closed neighborhood of $x_{\sigma^\smallfrown 0}$,
    disjoint from $F_{\sigma^\smallfrown 1}$.
  \end{proof}
 
  This ends the construction of the Hurewicz scheme. To
  finish the proof, we put $U=\bigcup_{\tau\in\ctr} U_\tau$,
  $C=\bigcap_{n<\omega}\bigcup_{\tau\in2^n} F_\tau$ and
  $c:\cantor\hookrightarrow C\subseteq X$ such that
  $c(x)\in\bigcap_{n<\omega} F_{x\upharpoonright n}$ for
  each $x\in\cantor$.
\end{proof}

\begin{proposition}\label{piececoncoanaonana}
  The $\sigma$-ideal $\I^f$ is $\coana$ on $\ana$.
\end{proposition}
\begin{proof}
  This follows from Corollary \ref{coanaonana} since the
  family of closed sets on which $f$ is continuous is
  hereditary and $\coana$.
\end{proof}

\begin{remark}
  Using Proposition \ref{piececon} we can explicitly write
  the formula defining the set of closed sets in $\I^f$.
  Let $\tilde K\subseteq \baire\times X$ be the universal
  closed set. Notice that $\tilde K_x\not\in \I^f$ if and
  only if
  \begin{eqnarray*}
    \exists U\subseteq\cantor\mbox{ open}\quad\exists c:\cantor\rightarrow
    \tilde K_x\mbox{ topological embedding}\qquad\qquad\\\big(c[\cantor]\cap f^{-1}[U]\mbox{ is
      dense in }c[\cantor]\big)\wedge \big(c[\cantor]\setminus
    f^{-1}[U]\mbox{ is dense in }c[\cantor]\big).    
  \end{eqnarray*}
  Indeed, the left-to-right implication follows from
  Proposition \ref{piececon} (when $G=X$).  The
  right-to-left implication holds because the set
  $f^{-1}[U]$ is an $\fsigma$ set which is dense and meager
  on $c[\cantor]$, therefore it cannot be a relative
  $\gdelta$ set on $c[\cantor]$. Hence, by the Jayne-Rogers
  theorem we have that $f$ is not piecewise continuous on
  $c[\cantor]$ and $\tilde K_x\not\in \I^f$.
  
  Now, the above formula is $\ana$. Indeed, it is routine to
  write a $\ana$ formula saying that $c:\cantor\rightarrow
  K_x$ is a topological embedding. The first clause of the
  conjunction can be written as
  $$\forall\tau\in\ctr\ \exists x\in[\tau]\quad f(c(x))\in
  U,$$ which is $\ana$, and analogously we can rewrite the
  second clause.
\end{remark}

If $G\subseteq X$ is a $\gdelta$ set and
$b:\baire\rightarrow G$ is a homeomorphism, then we call
$(b,G)$ \textit{a copy of the Baire space} and denote it by
$b:\baire\hookrightarrow G\subseteq X$.

\begin{proposition}\label{piececonksigma}
  For any $B\in \P_{\I^f}$ there is an $\I^f$-positive
  $\gdelta$ set $G\subseteq B$ and a copy of the Baire space
  $b:\baire\hookrightarrow G\subseteq X$ such that
  $$\I^f\har G=\{b[A]: A\subseteq\baire, A\in\ksigma\}.$$
\end{proposition}

\begin{proof}
  By Theorem \ref{Sol.dich} and the continuous reading of
  names we may assume that $B$ is of type $\gdelta$ and $f$
  is continuous on $B$.  Applying Proposition \ref{piececon}
  we get a copy of the Cantor space
  $c:\cantor\hookrightarrow C\subseteq X$ and an open set
  $U\subseteq\baire$ such that $f^{-1}[U]\cap C=c[\Q]$ and
  $C\setminus c[\Q]\subseteq B$. Let $G=C\setminus c[\Q]$.
  Via a natural homeomorphism of $\baire$ and
  $\cantor\setminus\Q$ we get a copy of the Baire space
  $b:\baire\hookrightarrow G\subseteq X$. Note that
  $C\not\in \I^f$ (by Theorem \ref{jrthm}, since
  $f^{-1}[U]\cap C$ is not $\gdelta$ in $C$) and hence also
  $G\not\in \I^f$.

  The $\sigma$-ideal $\I^f\har G$ is generated by the sets
  $D\cap G$ for $D\subseteq C$ closed such that $f\har D$ is
  continuous. The $\sigma$-ideal $\{b[A]: A\subseteq\baire,
  A\in\ksigma\}$ is generated by compact subsets of $G$. We
  need to prove that these two families generate the same
  $\sigma$-ideals on $G$.

  If $D\subseteq G$ is compact, then $D$ is closed in $C$
  and $f$ is continuous on $D$ because $f$ is continuous on
  $G$. Hence $D=D\cap G\in \I^f\har G$.

  If $D\subseteq C$ is such that $f$ is continuous on $D$,
  then $D\cap G= (f\har D)^{-1}[\baire\setminus U]$ is a
  closed in $D$ subset of $G$, therefore compact.

  This ends the proof.
\end{proof}

As an immediate consequence of Proposition
\ref{piececonksigma} we get the following corollary.

\begin{corollary}\label{piececonmiller}
  The forcing $\P_{\I^f}$ is equivalent to the Miller
  forcing.
\end{corollary}

Recall a theorem of Kechris, Louveau and Woodin
\cite[Theorem 7]{KLW}, which says that any coanalytic
$\sigma$-ideal of compact sets in a Polish space is either a
$\gdelta$ set, or else is $\coana$-complete. If $X$ is
compact, then $\I^f\cap K(X)$ is a coanalytic $\sigma$-ideal
of compact sets by Proposition \ref{piececoncoanaonana}.

\begin{proposition}
  $\I^f\cap F(X)$ is a $\coana$-complete set in $F(X)$.
\end{proposition}
\begin{proof}
  As in Proposition \ref{piececonksigma} take
  $c:\cantor\hookrightarrow C\subseteq X$ a copy of the
  Cantor space and $b:\baire\hookrightarrow G\subseteq X$ a
  copy of the Baire space, $G\subseteq C$ a dense $\gdelta$
  set in $C$. Recall that the Borel structure on $F(C)$ is
  induced from the topology of the hyperspace.

  It is well-known (see \cite[Exercise 27.9]{Kechris}) that the
  set $\ksigma\cap F(\baire)$ is a $\coana$-complete set in
  $F(\baire)$.  Let $\varphi:F(\baire)\rightarrow K(C)$ be
  the function $$F(\baire)\ni F\mapsto\overline{b[F]}\in
  K(C),$$ where $\overline{A}$ denotes the closure of $A$ in
  $C$ (for $A\subseteq G$). It is routine to check that
  $\varphi$ is Borel-measurable. By Proposition
  \ref{piececonksigma} we have $\varphi^{-1}[\I^f]=\ksigma$.
  This proves that $\I^f$ is $\coana$-complete.
\end{proof}

Piecewise continuity of functions from $\baire$ to $\baire$
has already been investigated from the game-theoretic point
of view. In \cite{vanWesep} Van \mbox{Wesep} introduced the
Backtrack Game $G_\mathrm{B}(g)$ for functions
$g:\baire\rightarrow\baire$. \mbox{Andretta} \cite[Theorem
21]{Andretta} characterized piecewise continuity of a
function $g$ in terms of existence of a winning strategy for
one of the players in the game $G_\mathrm{B}(g)$.

For a Borel not piecewise continuous function
$g:\baire\rightarrow\baire$, the Backtrack Game can be used
to define a fusion scheme for the $\sigma$-ideal $\I^g$. In
the remaining part of this section, we will show a very
natural fusion scheme for $\I_g$ when
$g:\cantor\rightarrow\cantor$ is Borel, not piecewise
continuous.

Recall that partial continuous functions from $\cantor$ to
$\cantor$ with closed domains can be coded by monotone
functions from $\ctree$ into $\ctree$ (see \cite[Section
2B]{Kechris}).

If $T\subseteq\ctree$ is a finite tree,
$m:T\rightarrow\ctree$ is a monotone function and
$n<\omega$, then we say that $(T,m)$ is a \textit{monotone
  function of height $n$} if $m(\tau)$ is of lenght $n$ for
each terminal node $\tau$ of $T$. We say that $(T_1,m_1)$
\textit{extends} $(T_0,m_0)$ if $T_1$ is an end-extension of
$T_0$ and $m_1\supseteq m_0$.

We define the game scheme $G_\pc$ as follows. In his $n$-th
move, Adam picks $\xi_n\in \omega^n$ such that
$\xi_n\supseteq \xi_{n-1}$ ($\xi_{-1}=\emptyset$). In her
$n$-th turn, Eve constructs a sequence of finite monotone
functions $\langle H^n_i: i<\omega\rangle$ such that
\begin{itemize}
\item $\forall^\infty i\ H^n_i=\emptyset$,
\item $H^n_i$ extends $H^{n-1}_i$,
\item $H^n_i$ is a monotone function of height $n$ whenever
  $H^n_i\not=\emptyset$.
\end{itemize}
In each play in $G_\pc$, for each $i<\omega$ we have that
$H_i=\bigcup_{n<\omega} H^n_i$ is a monotone function which
defines a partial continuous function $h_i$ with closed
domain (possibly empty).

Let $g:\baire\rightarrow\baire$ be a not piecewise
continuous function and $B\subseteq\baire$. The game
$G_\pc^g(B)$ is a game in the game scheme $G_\pc$ with the
following payoff set. Eve wins a play $p$ in $G_\pc^g(B)$ if
for $x=\bigcup_{n<\omega} \xi_n$ ($\xi_n$ is the $n$-th move
of Adam in $p$)
$$x\not\in B\quad\vee\quad\exists i\in\omega\ (x\in\dom(h_i) \wedge g(x)=h_i(x))$$
(where the functions $h_i$ are computed from Eve's moves as
above). Otherwise Adam wins $p$.

\begin{proposition}\label{piececongame}
  For any set $A\subseteq\baire$, Eve has a winning strategy
  in the game $G_\pc^g(A)$ if and only if $A\in \I^g$.
\end{proposition}
\begin{proof}
  If $A\in \I^g$ then there are closed sets
  $C_n\subseteq\cantor$ such that $A\subseteq\bigcup_n C_n$
  and $g\har C_n$ is continuous. Each function $g\har C_n$
  has its monotone function $G_n$ and Eve's strategy is
  simply to rewrite the $G_n$'s.

  On the other hand, suppose that there is a winning
  strategy for Eve and let $S$ be the tree of this strategy.
  The nodes of $S$ are determined by Adam's moves, so $S$ is
  isomorphic to $\ctree$. For $\tau\in T$ let
  $m^\tau_k:T^\tau_k\rightarrow\ctree$ be the monotone
  function $H_k$ defined by Eve in her last move of the
  partial play $\tau$. Denote by $H^\tau_k[\tau]$ the
  restriction of $m^\tau_k$ to $T^\tau_k(\tau)$. Put
  $G_k=\bigcup_{\tau\in T} H_k^\tau[\tau]$ and let $g_k$ be
  the partial continuous function with closed domain
  determined by the monotone function $G_k$. It follows from
  the fact that $S$ is winning for Eve, that $g\har
  A\subseteq\bigcup_n g_n$. This proves that $A\in \I^g$.
\end{proof}

\begin{corollary}\label{corpiececongame}
  If $B\subseteq\cantor$ is Borel and
  $g:\cantor\rightarrow\cantor$ is a Borel, not piecewise
  continuous function, then $B\in \I^g$ if and only if Eve
  has a winning strategy in $G_\pc^g(B)$.
\end{corollary}

\end{document}